\documentclass[a4paper,reqno,10pt,oneside]{amsart}

\usepackage{amsfonts,amssymb,latexsym,xspace,epsfig,graphics,color}
\usepackage{amsmath,enumerate,stmaryrd,xy,stackrel}
\usepackage[foot]{amsaddr}
\usepackage{indentfirst}
\usepackage[T1]{fontenc}
\usepackage{a4wide}
\usepackage{url}
\usepackage{tikz}
\usepackage{tkz-berge}
\usetikzlibrary{arrows,snakes,backgrounds}
\usepackage{color}
\usepackage{graphpap}
\usepackage{epsfig}
\usepackage{psfrag}
\usepackage{graphicx}

\theoremstyle{plain}
\newtheorem{teo}{Theorem}[section]
\newtheorem{cor}{Corollary}[section]
\newtheorem{lem}{Lemma}[section]
\newtheorem{prop}{Proposition}[section]
\theoremstyle{definition}
\newtheorem{defn}{Definition}[section]

\theoremstyle{remark}
\newtheorem{obs}{Remark}

\numberwithin{equation}{section}
\numberwithin{figure}{section}

\def\bbN{{\mathbb N}}
\def\bbE{{\mathbb E}}

\def\bbP{{\mathbb P}}
\def\GG{{\mathcal G}}
\def\VV{{\mathcal V}}
\def\EE{{\mathcal E}}

\def\SS{{\mathcal S}}
\def\RR{{\mathcal R}}
\def\CC{{\mathcal C}}
\def\NN{{\mathcal N}}

\def\YY{{\mathcal Y}}
\def\UU{{\mathcal U}}
\def\AA{{\mathcal A}}
\def\II{{\mathcal I}}

\providecommand{\be}{\begin{equation}}
  \providecommand{\ee}{\end{equation}}
\providecommand{\bea}{\begin{eqnarray}}
  \providecommand{\eea}{\end{eqnarray}}
\providecommand{\ba}{\begin{eqnarray}}
  \providecommand{\ea}{\end{eqnarray}}
\providecommand{\beas}{\begin{eqnarray*}}
  \providecommand{\eeas}{\end{eqnarray*}}

\providecommand{\beni}{\begin{equation*}}
  \providecommand{\eeni}{\end{equation*}}

\providecommand{\bw}{\begin{widetext}}
  \providecommand{\ew}{\end{widetext}}

\title[Maki-Thompson rumor model on a small-world network]{Phase transition for the Maki-Thompson rumour model\\ on a small-world network}
\date{}
\author{Elena Agliari$^1$}  
\address{$^1$Sapienza Università di Roma}
\author{Angelica Pachon$^2$} 
\address{$^2$Università di Torino}
\author{Pablo M. Rodriguez$^3$} 
\address{$^3$Universidade de São Paulo and Université Paris-Diderot}
\author{Flavia Tavani$^1$}


\begin{document}


\begin{abstract}
We consider the Maki-Thompson model for the stochastic propagation of a rumour within a population. We extend the original hypothesis of homogenously mixed population by allowing for a small-world network embedding the model. This structure is realized starting from a $k$-regular ring and by inserting, in the average, $c$ additional links in such a way that $k$ and $c$ are tuneable parameter for the population architecture. We prove that this system exhibits a transition between regimes of localization (where the final number of stiflers is at most logarithmic in the population size) and propagation (where the final number of stiflers grows algebraically with the population size) at a finite value of the network parameter $c$. A quantitative estimate for the critical value of $c$ is obtained via extensive numerical simulations.
\end{abstract}

\maketitle

\section{Introduction}
The spreading of information, diseases, fads or rumours within a community constitutes an important social phenomenon whose control is nowadays questioned and looked for (see e.g., \cite{spreading1,spreading2,spreading3}).
Remarkably, the evolution of the spreading process and the existence of a possible stationary state have been proved to depend sensibly on the underlying topology: One of the most significant examples is that spreading in an infinite-size scale-free network can involve a macroscopic fraction of the whole population regardless of its propagation rate \cite{vespignani-PRL2001,vespignani-EPL2002,moreno-PRE2003,newman-PRE2002}.

In the last decade, the striking proliferation of web-related transmission tools (e.g., `word-of-email' and `word-of-web') has prompted much attention on rumour (or rumour-like) propagation. As underlined in \cite{moreno-PhysA2007}, these transmission mechanisms are at the basis of viral marketing, of large-scale information dissemination and of fake-news amplification.
As a result, the possibility to make quantitative previsions on how many people will be involved by the rumour at a given time is of primal importance for companies as well as for governmental policies \cite{Kimmel-JBF2004,Kosfeld-JME2005,Agliari-IMA2010,Agliari-PRE2007}.

In this work we consider a standard model for rumour spreading, namely the Maki-Thompson (MT) model \cite{maki-1973}, and we adopt a small-world community structure feasible for a rigorous treatment.
More precisely, in the original formulation of the model, a closed and homogeneously mixed population is subdivided into three groups, referred to as ignorants, spreaders and stiflers, respectively. The firsts are unaware of the rumour, the seconds have heard it and actively spread it, and the thirds have heard the rumour but have ceased to spread it. The rumour is propagated through the population by pair-wise contacts: Any spreader attempts to pass the rumour to the other individual; in the case this other individual is an ignorant, it becomes a spreader, while in the other two cases the initiating spreader turns into a stifler. In a finite population the process will eventually reach a stack situation where individuals are either stiflers or ignorants. The MT model has been widely studied and, among the rigorous results available, we recall that Sudbury proved that, as the population size tends to infinity, the proportion of the population never hearing the rumour converges in probability to 0.2032 \cite{22}; Watson later derived the asymptotic normality of a suitably scaled version of this proportion \cite{23} and a corresponding large deviations principle, with an explicit formula for the rate function was found by Lebensztayn \cite{Lebensztayn-JMAA2015}; we also refer to \cite{sudbury-JAP1985, pittel-JAP1987, picard-JAP1994, noymer-JMS2001,gani-Env2000,lebensztayn/machado/rodriguez/2011a,lebensztayn/machado/rodriguez/2011b,raey} for further investigations on the model.

As we said, in the original version of the MT model there was no concern about the structure of the embedding community, but in more recent works, extensions to include also the underlying topology have been addressed \cite{moreno-PhysA2007}.
In fact, it is by now well-evidenced that social communities exhibit a relatively short diameter, namely one can go from one element of the network to another, arbitrary one passing by just a few others and this is also referred to as small-world (SW) property (incidentally, we notice that this kind of property is also found in other kinds of networks, such as the Internet, the World Wide Web, metabolic and protein interaction networks and food webs) \cite{Dorogovstev-Rev2002,Newman-2010}. The SW feature has been shown to improve the performance of many dynamical processes as compared to regular lattices; this is a direct consequence of the existence of key shortcuts that speed up the communication between otherwise distant nodes and of the shorter path length among any two nodes on the network \cite{Barrat-2012}. This feature should therefore be accounted for toward a more realistic description of the process.

Here we embed the model in a small-world architecture where the number of neighbours is a tuneable parameter: starting from a $k$-regular ring, we randomly insert additional links in such a way that the mean number of neighbours featured by any node is $2k+c$. We study the MT on this class of graphs by exploiting a rigorous approach, a mean-field like approach and numerical simulations.
\newline
We are especially interested in the asymptotic behavior of the number of stiflers which provides an estimate for the extent of the rumour penetration and also an estimate for the time span of the spreading.
We prove that the model exhibits a transition between regimes of localization (where the final number of stiflers is at most logarithmic in the population size) and propagation (where the final number of stiflers grows algebraically with the population size) at a finite value of the network parameter $c$. An estimate of the final number of stiflers is then obtained by adopting a mean-field approximation which is based on the assumption of negligeable correlations between the state of neighbouring individuals and it is therefore expected to hold when the network is most homogeneous (i.e., large $c$). Finally, by means of extensive simulations we check the overall scenario and we find a quantitative estimate for the critical thresholds where the two propagation regimes emerge.

We stress that the existence of two regimes was already evidenced in \cite{zanette} for the MT model on a small world via numerical investigations, while, to the best of our knolwedge, our result is the first example of rigorous treatment.

The remaining of the paper is structured as follows: in Sec. 2 we describe the topology considered, we review the rules for the rumour propagation and we state our main theorem on the existence of two regimes, where the final number of removed nodes remain localized in a ``small'' subset of the population and where the final number of removed nodes grows algebraically with the population size, respectively. Next, in Sec. 3 we approach the problem from a mean-field perspective, getting an estimate for the final number of removed nodes. Then, in Sec. 4 we present numerical results which corroborate the analytical investigation and allow a quantitative estimate for the model parameter corresponding to the threshold between the regimes mentioned above. Finally, Sec. 5 contains discussions and outlooks. The proofs of the theorem and the technical details concerning numerical simulations are collected in Appendices A and B, respectively.

\section{The model}\label{S: The model}
Given  $p\in(0,1)$ and $k,n\in\bbN$, $k<n$, consider the Newman-Watts small-world random graph $\GG (n,k,p)$ described as follows (see e.g., \cite{NW1,NW2}):
\smallskip 
\begin{enumerate}
\item Construct a $k$-regular ring lattice of size $n$, which is a ring with nodes $[n]:=\{1,2,\ldots,n\}$, each connected to its $2k$ nearest neighbors, i.e., the nearest $k$ neighbors clockwise and counterclockwise. In other words, there is an edge between nodes $i$ and $j$ if, and only if,  $0<|i-j| \leq k$ $\mod n$.  

\smallskip 
\item For every possible pair of nodes $i,j$, with $|i-j|>k$ $\mod n$, add an edge connecting them with probability $p$, independently one from another. 
\end{enumerate}

\begin{figure}[ht]\label{fig:swrg}
\begin{center}
\SetVertexNoLabel
\begin{tikzpicture}[scale=0.8]
\GraphInit[vstyle=Simple]
    \SetGraphUnit{1.5}
    \SetUpEdge[style={-,bend right=12}]
    \tikzset{VertexStyle/.style = {shape = circle,fill = black,minimum size = 2pt,inner sep=1.5pt}}
  \begin{scope}[xshift=12cm]
    \grSQCycle[prefix=a,RA=3]{30}
  \end{scope}
  \Edges[style={bend right=0}](a1,a10,a22)
    \Edges[style={bend right=0}](a17,a28)
    \Edges[style={bend right=0}](a13,a29)
    \Edges[style={bend right=0}](a14,a2)
\end{tikzpicture}
\end{center}
\caption{The small-world random graph $\GG(n,k,p)$, $k=2$.}
\end{figure}
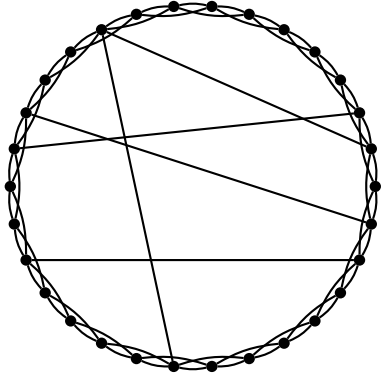

We will say that the edges generated in (1) are the `local edges' while the edges generated in (2) are the `long-range edges' or `shortcuts'. Furthermore, for any $i\in [n]$, let $\NN_{l}(i)$ denote the set of vertices in $[n]$ connected to $i$ through a local edge. Call these vertices the local neighbors of $i$.  Analogously,  for any $i\in [n]$, let $\NN_{s}(i)$ denote the set of vertices in $[n]$ connected to $i$ through a shortcut, and call them the long-range neighbors of $i$. Observe that, differently as $\NN_{l}(i)$, $\NN_{s}(i)$ is a random set.

Now consider a realization of  $\GG(n,k,p)$, say $G_n$, and  run the following  rumour spreading process on $G_n$. Let the vertices in $G_n$ represent individuals of a population, which is subdivided into ignorants, spreaders and stiflers. Assume at time $t=0$ there are   only one spreader and $n-1$ ignorants. At time $t+1$ a spreader is chosen uniformly at random from the current set of spreaders, and it attempts to transmit the information to one of its neighbors. The neighbor is selected with uniform probability from the set of neighbors of the spreader. If the selected neighbor is an ignorant, it becomes a spreader, otherwise the spreader who is trying to transmit the information, becomes a stifler. For simplicity in the exposition, we will refer to a stifler as being a removed individual. We refer to this model as the Maki-Thompson rumour process on $\GG(n,k,p)$.

\subsection{Construction of the model}
\label{sec:construction}

Observe that there are two layers of randomness incorporated in the model, the randomness coming from the underlying random graph $G_n$, and the randomness coming from the rumour dynamic running on $G_n$. The idea now is to  consider a probability space where we can grow the rumour process and the random graph in tandem. For this, we consider shortcuts of $G_n$ only at those times when the rumour process attempts a new transmission from a vertex whose neighborhood structure is not yet determined. The construction is similar to the one proposed by Lalley and Su (2015)  for the contact process on random regular graphs, see \cite[Section 3.2]{lalley/su/2015}.      

\smallskip
As before, let $p\in(0,1)$ and $k,n\in\bbN$, $k<n$ and consider the $k$-regular ring lattice of size $n$ as in (1). Thus, for any $i\in [n]$, $\NN_{l}(i)$ corresponds to the set of local neighbors of $i$.
Furthermore, let  $(\Omega, \mathcal{F},\bbP)$ be a probability space  where the following random objects are independent and well defined  for any $n\geq 1$. 

\begin{itemize}
\item $\{I_{ij}: i,j\in[n], i\neq j\}$ i.i.d. Bernoulli random variables of parameter $p$, where  $I_{ij}=I_{ji}, \forall i,j\in[n], i\neq j$.\\[.1cm]
\item $\mathcal{N}_s(i):=\{j\in[n]\setminus \mathcal{N}_{l}(i): I_{ij}=1\}$, for each $i\in [n]$, and\\[.1cm]
\item $\{U_{i}^{1},U_{i}^2,\ldots\}$ i.i.d. random variables with uniform law on $\mathcal{N}(i):=\mathcal{N}_{s}(i)\cup \mathcal{N}_{l}(i)$.
\end{itemize}

\smallskip
The random variables defined above depend on $n$  but for simplicity we suppress this dependence in our notation.  At each time $t$, for $t\geq 0$, we construct simultaneously two random objects, which  essentially  define the model. The first is a random variable $\eta_t\in  \{0,1,2\}^{[n]}$, that gives us the information about the possible configurations of the rumour process at time $t$, where $0,1,2$ are used, respectively, to represent an ignorant, a spreader and a removed (stifler) vertex. The second, is a random set $\EE_t$ of determined shortcuts up to time $t$. 

Assume that the rumour process starts at time $t=0$ from vertex $v_0$. We denote the set of ignorant, spreader and removed vertices at time $t$, $t\geq 0$, by $\II(t), \SS(t)$ and $\RR(t)$, respectively, and we let 
$$\mathcal{I}(0)=[n]\setminus\{v_0\}, \;\;\;\mathcal{S}(0)=\{v_0\}, \;\;\;\text{and}\;\;\;\mathcal{R}(0)=\emptyset.$$

Let $\mathcal{E}_0 = \emptyset$. At time $t=1$ we determine the shortcuts which are incident to $v_0$, that is, we let $\mathcal{E}_1= \{(v_0,j)\in \{v_0\}\times [n]\setminus \mathcal{N}_{l}(v_0):I_{v_0 j}=1\}$. When $v_0$ attempts to transmit the information, one of all the possible edges incident to $v_0$, say $(v_0,v_1)$, is selected according to $U_{v_0}^1$. Then we let  $\mathcal{I}(1)=[n]\setminus\{v_0,v_1\}$, $\mathcal{S}(1)=\{v_0,v_1\}$ and $\mathcal{R}(1)=\emptyset$, and define a random variable $U_1$ uniformly distributed on $\SS(1)$. At any time $t+1$, a vertex of $\SS(t)$, say $v$, is selected according to $U_t$. Then, we let $\mathcal{E}_{t+1}=\mathcal{E}_{t}\cup \{(v,j)\in \{v\}\times [n]\setminus \mathcal{N}_{l}(v):I_{v j}=1\}$. Here we point out that, although there exists a chance to select the same vertex $v$ on different times of this construction, the set $\mathcal{E}_{t+1}$ does not count twice the same shortcuts. Now, continuing with the procedure, one of all the possible edges incident to $v$, say $(v,v_\ell)$, is selected according to $U_{v}^{m}$, where $m$ denotes the $m$-th time when $v$ is selected to inform.   Then,

\begin{equation}\label{eq:transitions}
\begin{array}{lcl}
\text{ if } v_{\ell} \in \mathcal{I}(t) & \rightarrow &
\begin{cases}
\mathcal{I}(t+1)  = \mathcal{I}(t)\setminus \{v_\ell\}\\
\mathcal{S}(t+1) = \mathcal{S}(t)\cup \{v_\ell\} \\
\mathcal{R}(t+1) = \mathcal{R}(t), \text{ and }
\end{cases}\\
&&\\
\text{ if } v_{\ell} \in \mathcal{S}(t)\cup \mathcal{R}(t) & \rightarrow &
\begin{cases}
\mathcal{I}(t+1) = I(t)\\
\mathcal{S}(t+1)=\mathcal{S}(t)\setminus \{v_\ell\} \\
\mathcal{R}(t+1)=\mathcal{R}(t)\cup \{v_\ell\}.
\end{cases}
\end{array}
\end{equation}

\smallskip
In addition, define a new random variable $U_{t+1}$ uniformly distributed on $\SS(t+1)$ and continue this process up to the absorption time 
$$\tau_n:=\inf\{t>0: \SS(t)=\emptyset\}.$$ 
We take the sequence $\{U_1,U_2,\dots, U_{\tau_n}\}$ as a sequence of random variables independent of the collections of uniform random variables $U_i^m$ defined above. 

\smallskip

Simultaneously to the process $(\EE_{t})_{t=0}^{\tau_n}$, we also define  the process $(\eta_{t})_{t=0}^{\tau_n} \in \{0,1,2\}^{[n]}$, such  that, for each $i\in [n]$, and $t\geq 0$,
$$\eta_t(i):=
\begin{cases}
0,& \text{ if }i\in \mathcal{I}(t),\\[.1cm]
1,& \text{ if }i\in \mathcal{S}(t),\\[.1cm]
2,& \text{ if }i\in \mathcal{R}(t).
\end{cases}
$$

Note that $\mathcal{E}_{\tau_n}$ may not contain all the shortcuts that a realization of $\GG(n,k,p)$, say $G_n$, contains. In this case we add the remaining shortcuts edges in $\mathcal{E}_{\tau_n}$, according to the remaining random variables $I_{ij}$, to complete the random graph $G_n$. In this way we may obtain from $\mathcal{E}_{\tau_n}$ the random graph $G_n$. Therefore, the Maki-Thompson rumour process on $\mathcal{G}(n,k,p)$ is the pair $\left(G_n, (\eta_t)_{t=0}^{\tau_n}\right)$.

\subsection{Main Results}\label{mainresults}
Let $I(t)=|\mathcal{I}(t)|$, $S(t)=|\mathcal{S}(t)|$ and $R(t)=|\mathcal{R}(t)|$ be the number of ignorant, spreader and removed vertices at time $t$.
We are interested in the asymptotic behavior of the number of removed vertices at the absorption time $\tau_n$. More precisely, if we rewrite the absorption time of the process as $\tau_n:=\inf\{t>0: S(t)=0\}$ then, we are interested in the behavior of $R(\tau_n)$ as $n\rightarrow \infty$. 

The main result of our work states the existence of a phase transition regarding a major propagation, or not, of the rumour in the following sense.

\begin{defn}
Consider the Maki-Thompson rumour process on $\mathcal{G}(n,k,p)$. A {\it major propagation} is said to occur if there exists $0<\gamma\leq 1$, such that the rumour is propagated to at least $O(n^{\gamma})$ individuals. On the other hand, a {\it minor propagation} is said to occur if the rumour is propagated to at most $\ln n$ individuals. 
\end{defn} 

\smallskip
\begin{teo}\label{teor:phaset}
Consider the Maki-Thompson rumour process on $\GG (n,k,p)$ with $p=c/(n-2k-1)$,  $c$ a constant independent of $n$ and $k\geq 1$. Then there exist $c_1(k),c_2(k)\in (0,\infty)$ such that

\begin{enumerate}
\item[(i)] if $c<c_1(k)$, then a minor propagation occurs with high probability.
\smallskip
\item[(ii)] if $c>c_2(k)$, then a major propagation occurs with asymptotic non-zero probability.
\end{enumerate}
\end{teo}

\smallskip

\smallskip
\begin{obs}Theorem \ref{teor:phaset} reveals the existence of two well-differentiated regimes that we refer as the subcritical regime, i.e. minor propagation (i), and as the supercritical regime, i.e. major propagation (ii), of the rumour process. This agrees with similar results, on a related small-world network, reported by Zanette in \cite{zanette,zanette02}. In the subcritical regime we shall prove that the final number of removed vertices remains localized in a ``small'' subset of the population. In other words, we shall prove that $\lim_{n\to\infty} \mathbb{P}(R(\tau_n)<\ln n)=1$. On the other hand, in the supercritical regime, we shall prove that the final number of removed vertices grows with $n$, and this growth is larger than a $O(n^{\gamma})$ function, for some $\gamma$ such that $0<\gamma<1$. Here we obtain that $\liminf_{n\to \infty} \mathbb{P}(R(\tau_n)>O(n^{\gamma}))>0$.

\end{obs}

\smallskip
As a direct consequence of the construction of the model, we have the following relation between the absorption time of the process and the total number of removed vertices at that time.

\smallskip
\begin{prop}\label{Propo3.1}
Consider the Maki-Thompson rumour process on $\GG (n,k,p)$ with $p\in (0,1)$ and $k\geq 1$. Then, for any $n\geq 1$
\begin{equation}\label{eq:tauR}
\tau_n = 2 R(\tau_n) -1.
\end{equation} 
\end{prop}


\section{Approximation analysis: mean-field equations}

Theorem \ref{teor:phaset} rigorously states the existence of a phase transition regarding a major propagation, or not, of the rumour over the population with positive probability. Moreover, this phase transition depends on the value of parameter $c$, which may be interpreted as a measure of the the small-network randomness of the underlying graph. The purpose of this section, is to conjecture other aspects regarding the behavior of the rumour process as a function of parameters $k$ and $c$. In Subsection \ref{ssec:meanfield}, we give an approximation for the size of the removed vertices set at the end of the process in the supercritical regime. We shall accomplish this by defining a time-rescaled process, which is absorbed at the same state than the Maki-Thompson rumour process. Then, we shall derive and analyze a set of mean-field equations for the new process. 
Later, in Section \ref{sec:numerics}, we report numerical evidence for the localization of the critical value, i.e., the value of $c$ where phase transition occurs. 

\subsection{Time-rescaled process}
\label{ssec:meanfield}

We consider two time-changes in the construction of the Maki-Thompson rumour process $(\eta_{t})_{t\geq 0}$ on $\GG(n,k,p)$. For each $n\geq 1$, we assume that the consecutive times of the discrete time process are separated by a fixed time interval of size $1/(2k+c)n$. On the other hand, we assume that at each discrete time of the process, say $t/(2k+c)n$, any vertex of the whole population may be chosen at random to propagate the information. If the chosen vertex $v$ is a spreader then, one of all the possible edges incident to $v$, say $(v,v_\ell)$, is selected according to  $U_v^m$ and transitions in the rumour process occur according to \eqref{eq:transitions} (the random variable $U_{v}^{m}$ is the  one from the sequence of i.i.d. random variables, with uniform law on $\NN(i)$, used in Section \ref{sec:construction},  where $m$ denotes the $m$-th time when $v$ is selected to inform).
Otherwise, nothing happens and a new vertex's choice is performed at time $(t+1)/(2k+c)n$. In other words, we are replacing the sequence of uniform random variables $\{U_1,U_2,\ldots\}$ used in Section \ref{sec:construction}, by a sequence of i.i.d. random variables with common uniform law on $[n]$. We will abuse notation and identify both sequences with the same letters. Reproducing the construction of Section \ref{sec:construction} with these changes we obtain a time-rescaled process $(\eta_{t}^{r})_{t\geq 0}$ on $\GG(n,k,p)$. Moreover, if we define $\tau_n^{r}:=\inf\{t>0:S^{r}(t)=0\}$, it is not difficult to see that $R(\tau_n)=R^{r}(\tau_n^{r})$.

\subsubsection{Mean field equations}

In the sequel, we will derive a mean-field approximation for the time-rescaled process, which will allow  us to obtain an approximation for the final proportion of removed vertices on the Maki-Thompson rumour process. For any time $t\geq 0$, and $i \in [n]$, let
 \begin{equation} \label{probabilities}
\begin{split}
x_i(t) &:=P(\eta_t^r(i)=0),\\
y_i(t) &:=P(\eta_t^r(i)=1),\\
z_i(t) &:=P(\eta_t^r(i)=2).
\end{split}
\end{equation}
First, we shall approximate the functions $x_i(t), y_i(t),$ and $z_i(t)$, by the solutions of a system of ordinary differential equations. Let $h:=h(n)=1/(2k+c)n$, and observe that 
\begin{equation}\label{eq:firstxi}
x_i(t+h)=\left(1-P(\eta_{t+h}^r(i)=1|\eta_{t}^r(i)=0)\right) x_i(t).
\end{equation}
In addition, note that 
 \begin{equation}\label{eq:secxi}
P(\eta_{t+h}^r(i)=1|\eta_{t}^r(i)=0))=P(A_{i}(t+h)|\eta_{t}^r(i)=0),
\end{equation}
where
$$A_i(t+h):=\bigcup_{j\in \NN(i)}\{\eta_{t}^r(j)=1,U_{t} =j,U_{j}^{m}=i\}.$$
Then, 
\begin{equation}
\begin{aligned}
P(A_{i}(t&+h) |\eta_{t}^r(i)=0) =\\[.2cm]
=&  \sum_{\ell =0}^{n-2k-1} P(A_{i}(t+h)|\eta_{t}^r(i)=0, |\NN_s(i)|=\ell)\, P\left(|\NN_s(i)|=\ell | \eta_{t}^r(i)=0\right)\nonumber\\[.2cm]
= &  \sum_{\ell =0}^{n-2k-1} P\left(\bigcup_{j\in \NN(i)}\{\eta_{t}^r(j)=1,U_{t} =j,U_{j}^{m}=i\} \Big| \eta_{t}^r(i)=0, |\NN_s(i)|=\ell\right)\, P\left(|\NN_s(i)|=\ell\right)\\[.2cm]
= & \sum_{\ell =0}^{n-2k-1} (2k+\ell) P\left( \eta_{t}^r(j)=1,U_{t} =j,U_{j}^{m}=i\} \Big| \eta_{t}^r(i)=0, |\NN_s(i)|=\ell\right)\, P\left(|\NN_s(i)|=\ell\right)
\end{aligned}
\end{equation}

The mean-field procedure is an approximation method which consists, when we analyze the effect of the neighbors on a given vertex $i$, in assuming that vertex $i$ is not influenced in some relevant way. This method will allows to get a workable expression for \eqref{eq:secxi}, and therefore for \eqref{probabilities}. In other words, assume, if delicacy permits, independence between vertices and states. This, together with the previous computations leads to 
\begin{equation}\label{eq:Ai}
P(A_{i}(t+h) |\eta_{t}^r(i)=0) \stackbin{mf}{=}  \frac{(2k+c) \, \alpha}{n}\, y_j (t), 
\end{equation}
where notation $\stackbin{mf}{=}$ is to remark that such expression has been obtained after the mean-field assumption, and the final expression in \eqref{eq:Ai} is due to $y_j(t):=P(\eta_{t}^r(j)=1)$, $P(U_{t} =j)=1/n$, $\alpha:=P(U_{j}^{m}=i)=\mathbb{E}\left(\left\{X+2k\right\}^{-1}\right)$ where $X\sim Poisson(c)$, and 
$$\sum_{\ell =0}^{n-2k-1} (2k+\ell) \, P\left(|\NN_s(i)|=\ell\right) = (2k+c).$$

By \eqref{eq:firstxi}, \eqref{eq:secxi}, and \eqref{eq:Ai}, we obtain
$$x_i(t+h)- x_i(t) \stackbin{mf}{=} -  \frac{(2k+c) \, \alpha}{n}\, y_j (t) x_i(t),$$
which may be re-written as
\begin{equation}\label{eq:de}
\frac{x(t+h)- x(t)}{h} \stackbin{mf}{=} -  (2k+c)^2 \, \alpha \, y (t) x(t),
\end{equation} 
where $h=1/(2k+c)n$, and we have removed the dependence on $i,j$ in the notation because we are assuming that all the vertices have the same probabilistic behavior. Finally, notice that taking $n$ sufficient large, \eqref{eq:de} is approaching to
$$x^{\prime}(t) = - (2k+c)^2 \,\alpha\, x(t)\, y(t).$$

\noindent
Similar arguments lead to the following mean-field approximation for the probabilities in \eqref{probabilities}, 

\begin{equation}
\begin{cases}\label{eq:systemMF}
x^{\prime}(t) = - (2k+c)^2 \,\alpha\, x(t)\, y(t), \\[0.1cm]
y^{\prime}(t) = (2k+c)^2 \,\alpha \,x(t) \,y(t) - (2k+c)\,(1-\alpha) \,(y(t) +z(t))\,y(t) - (2k+c)\, \alpha \,y(t),\\[0.1cm]
z^{\prime}(t) =  (2k+c)\, (1-\alpha) \,(y(t) +z(t))\,y(t) + (2k+c)\, \alpha\, y(t),
\end{cases}
\end{equation}

\smallskip
\noindent
where $\alpha:=\mathbb{E}\left(\left\{X+2k\right\}^{-1}\right)$, with $X\sim Poisson(c)$, and we assume $x(0) = 1$, and $y(0) = z(0) = 0$.  

\bigskip
\begin{obs}We point out that whether $X\equiv c$, which implies $\alpha = 1/(2k+c)$, we obtain from \eqref{eq:systemMF} the system

\begin{equation}
\begin{cases}\label{eq:systemMF2}
x^{\prime}(t) = - (2k+c)  x(t)\, y(t), \\[0.1cm]
y^{\prime}(t) = (2k+c)  \,x(t) \,y(t) - (2k+c -1)  \,(y(t) +z(t))\,y(t) -  \,y(t),\\[0.1cm]
z^{\prime}(t) =  (2k+c -1) \,(y(t) +z(t))\,y(t) + \, y(t),
\end{cases}
\end{equation}

\smallskip
\noindent
which is related to the mean-field approximation given by Moreno et al. \cite{MNP-PRE2004}, for the Maki-Thompson rumour model on homogeneous networks. Observe that homogeneity is present in our model, not in the strict sense but in the probabilistic one.
\end{obs}

\bigskip
\subsubsection{Final proportion of removed vertices}

An approximation for the final proportion of removed vertices for the time-rescaled process $(\eta_{t}^r)_{t\geq 0}$, for $n$ sufficiently large, may be obtained by taking the limit $z_\infty:=\lim_{t\to\infty} z(t),$ where $z(t)$ is solution of \eqref{eq:systemMF}. By construction, we have that $R(\tau_n)=R^{r}(\tau_n^{r})$, and hence $z_\infty$ will be our approximation for the Maki-Thompson rumour process $(\eta_{t})_{t\geq 0}$ on $\GG(n,k,p)$ as well. Introducing a time-change in system \eqref{eq:systemMF} we may obtain a transcendental equation for $z_\infty$, i.e.,
\begin{equation}
z_\infty = 1 - e^{-\left\{(2k+c)\, \alpha\, + 1 - \alpha \right\} z_\infty}.
\end{equation}
Moreover, consider the Lambert $W$ function defined as the multivalued inverse of the function $x\mapsto e^x$. If $W_0$ denotes its principal branch, which is the branch satisfying $W(x)\geq -1$, the proportion $z_{\infty}$ may be explicitly written as 
\begin{equation} \label{eq:MF}
z_\infty = 1 + \frac{W_0\left(-\left[(2k+c)\alpha + 1 -\alpha\right] e^{-\{(2k+c)\alpha + 1 -\alpha\}}\right)}{(2k+c)\alpha + 1 -\alpha}
\end{equation}
by noting that $1/e < -\left[(2k+c)\alpha + 1 -\alpha\right] e^{-\{(2k+c)\alpha + 1 -\alpha\}} < 0$. We refer the reader to Corless et al. \cite{corless} for a review of results and applications about the Lambert function.

\section{Numerical results \label{sec:numerics}}

In this section we present the numerical results obtained by simulating the Maki-Thompson model defined in Sec.~\ref{sec:construction}. We are especially interested in measuring the final number of removed individuals $R(\tau_n)$ as the parameters $c$ and $k$ are varied. Thus, in order to highlight the dependence on the parameters, the average (vide infra) final number of removed individuals shall be referred to as $R(c,k)$.

Before presenting the numerical results it is worth spending a few words on how simulations were run. \\
First, we build a proper realization $G_n$ of $\mathcal{G}(n,k,p)$, where the coordination number (i.e., the number of links stemming from a node) of an arbitrary node is due to two contributions: a deterministic one equal to $2k$ and a stochastic one which, in the average, is equal to $c$ and independent of the node. In order to check that simulations fit the theoretical model, we test that the stochastic contribution to the coordination number is distributed according to a binomial distribution $Binomial (n-2k-1,p)$.\\
Then, we implement the Maki-Thompson dynamics previously described (see \ref{eq:transitions}). More precisely, we initially extract randomly a node, which will play as the original spreader, while the remaining $n-1$ nodes are ignorant; as a consequence $S(0)=1$ while $I(0)=n-1$ and $R(0)$=0. Then, in the next time steps $t=1, 2, ...$, we repeat sequentially the following rules: \\
We consider the set of spreaders $S(t-1)$, and we extract randomly a node $i \in S(t-1)$;\\
From the set of neighbours (both local and non-local) $\mathcal{N}(i)$ of $i$ we extract randomly a node $j$; \\
If the neighbour $j$ is ignorant, then $j$ turns into a spreader and, accordingly, $S(t+1)=S(t)+1$ while $I(t+1)=I(t)-1$, otherwise, if the neighbour $j$ is either spreader or stifler, then $i$ turns into a stifler and, accordingly, $S(t+1)=S(t)-1$ while $R(t+1)=R(t)+1$.\\
This procedure is repeated as long as there are spreaders among the individuals (i.e., as long as $S(t) > 0$). 
The final time is referred to as $\tau_n$, namely $S(t) \equiv 0, \forall t\geq \tau_n$. On the other hand, as long as $t<\tau_n$, the number of removed individuals $R(t)$ is monotonically increasing with time. \\
When the simulation stops we check that the final time $\tau_n$ is related to the final number of removed nodes by $\tau_n= 2R(\tau_n)-1$, as proved in Proposition \ref{Propo3.1},
and we collect the final value $R(\tau_n)$ for the number of removed individuals.
Then, resuming the initial setting, the simulation is repeated $M$ times and, for each realization we get an estimate for $R(\tau_n)$.
\newline
Next, we build up another realization for the graph $\mathcal{G}(n,k,p)$ and, again, we repeat $M$ realizations for the dynamic process.
We repeat this operation $L$ times in such a way that we finally have $M \times L$ estimates for $R(\tau_n)$ which are averaged together. The average over the different realizations of the dynamic process and over the different realizations of the underlying graph is the quantity $R(c,k)$ used in the analysis described in the following subsections. We stress that $M \times L$ simulations are run in order to get an estimate for the final number of removed individuals, given the parameters $c$ and $k$. Since we want to analyze $R$ as a function of $c$ and $k$, just as many simulations are performed for each choice of the parameter set. 
In general, for the analysis performed in this work we used $M=10^5$ and $L=10$.\\
For further insights on averages we refer to Appendix B.

The simulation outcomes, distinguishing between the case $k=1$ (Sec.~\ref{ssec:K1}) and $k>1$ (Sec.~\ref{ssec:K2}) are now presented.

\subsection{The case $k=1$ \label{ssec:K1}}
From the theoretical results proved in the Theorem \ref{teor:phaset}, we expect to highlight two different regimes in the final configuration of the system: when $c$ is relatively small $R(c,k)$ should scale slowly (i.e., at most logarithmically) with the system size $n$, while when $c$ is relatively large $R(c,k)$ should scale faster (i.e., algebraically) with the system size $n$. Now, given the right scaling, say $R(c,k) \sim n^{\gamma}$, by plotting the ratio $R(c,k)/n^{\gamma}$ as a function of $c$, the data  pertaining to different sizes should collapse and the value of $c$ where this collapse starts (ends) provides a numerical estimate for $c_2$ ($c_1$), where $c_i:=c_i(k)$, for $i=1,2$, are the stated in Theorem \ref{teor:phaset}. In Fig.~\ref{collapse1} we show the collapse of data obtained for $n=3200, 6400, 12800, 25600$ and $c \in [0,10]$. More precisely, in the left panel we highlight a collapse for $R(c,k)$ in the lower region of $c$ (i.e., $c \leq 0.6$), while in the right panel we highlight a collapse for $R(c,k)/n$ in the upper region of $c$ (i.e., $c \geq 1.4$). This analysis allows us to get an estimate for the qualitative behavior of $R(c,k)$ versus $n$ and for the thresholds $c_1$ and $c_2$. In fact, in the major propagation regime we expect that $R(c,k) \sim n$, namely $\gamma =1$ and $c_2 \approx 1.4$. On the other hand, in the minor propagation regime, we expect that $R(c,k)$ remains fine, namely it does not scale with $n$, and $c_1 \approx 0.6$.
\newline
Finally, as $c$ gets very large $R(c, k)/n$ tends to saturate to a value around $0.8$. As expected, this value is close to the asymptotic result obtained for fully-connected models in fact, as $c$ gets larger the underlying graph $\mathcal{G}(n,k,p)$ approaches the complete graph $K_n$. 


\begin{figure}[tb]
	\centering
	\includegraphics[width=10cm]{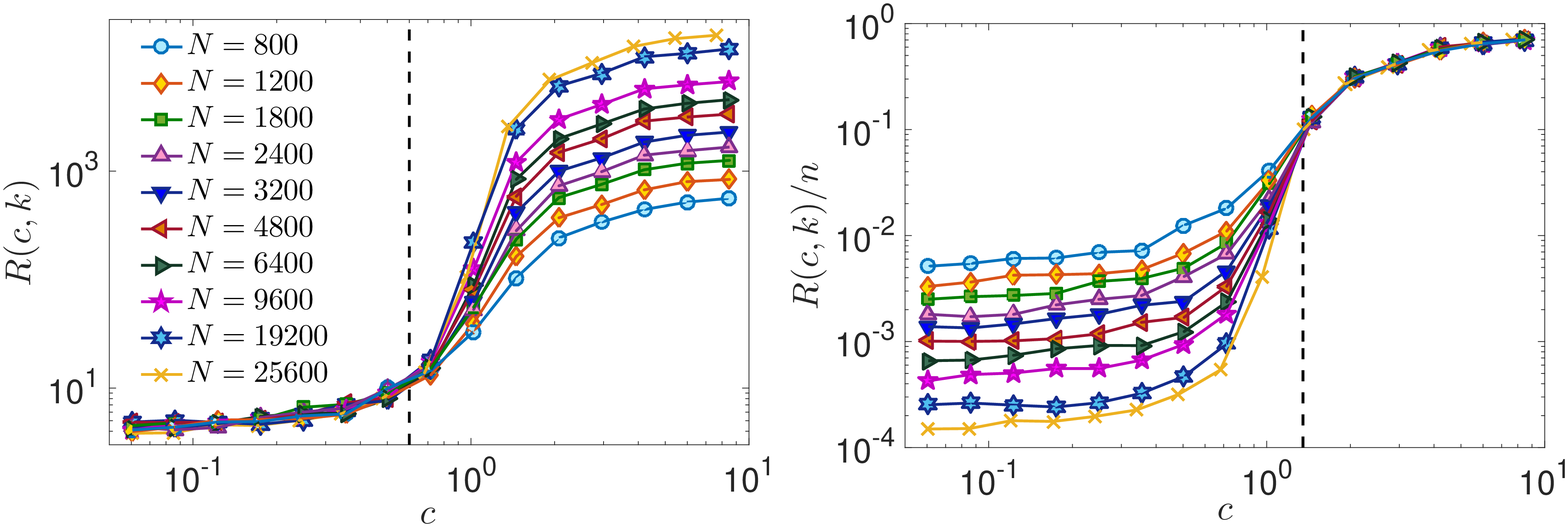}
	\caption{The average number of removed units $R(c,k)$ (left panel) and the average fraction of removed units $R(c, k)/n$ are shown versus $c$; the value of $k$ is set equal to $1$ and several choices for the system size are considered and plotted in different colors and symbols as explained by the legend. Notice that curves for $R(c,1)$ pertaining to different size collapse for relatively small values of $c$, while, in order to get an analogous collapse also in the region of relatively large $c$ we need to normalize $R(c, 1)$ with respect to $n$. The vertical dashed lines indicate the thresholds $c_1$ and $c_2$ below which and above which, respectively, the collapse occurs. Notice the logarithmic scale in both panels.}\label{collapse1}
\end{figure}

 
In order to further characterize the two regimes of propagation we study the probability distribution for $R(c,k)$. More precisely, we introduce $p_1(r;k,c)$, which represents the probability that $R(c,k)$ is equal to $r$, and, similarly,  $p_2(r;k,c)$, which represents the probability that $R(c,k)/n=r$. As shown in Fig.~\ref{Distr}, these distributions provide a good representation for the subcritical (left panel) and for the supercritical (right panel) regime, respectively. In fact, the histograms corresponding to the same choice of $c$, but to different size are again overlapped suggesting that the scalings adopted (i.e., the raw number $R(c,k)$ and the linearly normalized number $R(c,k)/n$) are correct.  
Notice that $p_1(r;k,c)$ displays an approximately exponential decay with $r$, while $p_2(r;k,c)$ is bell-shaped (the peak on small $r$ apart) and its width decreases with $c$.\\
When $c \gg c_2$, $p_2(r;k,c)$ is sharply peaked at a value approaching $0.7$, close to the asymptotic value $0.8$ expected for a fully connected network.

 
\begin{figure}[tb]
	\centering
	\includegraphics[width=9.2cm]{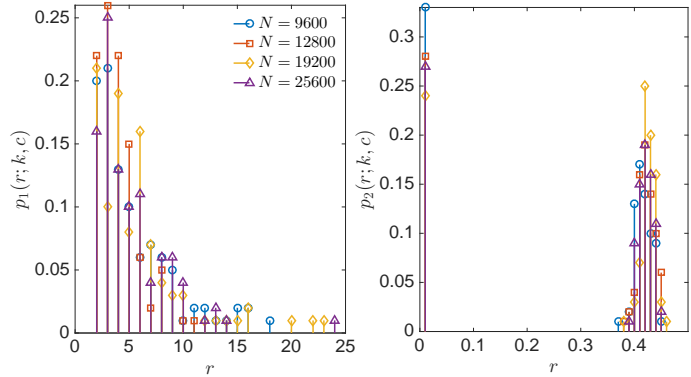}
	\caption{In these panels we show the histograms for the raw data of the average number of removed units $R(c, k)$ (left) and for the average fraction of removed units $R(c, k)/n$ (right), for different choices of $n$, as reported in the legend, while the value of $k$ is fixed and set equal to $1$. As for $c$, we selected two values corresponding to the subcritical and supercritical region, respectively, that is, $c=0.1$ (left) and $c=2$ (right). According to the results shown in Fig.~\ref{collapse1} for the average values, we expect that the distributions for $R(c, 1)$ (respectively $R(c, 1)/n$) pertaining to different system sizes exhibit the same mean as long as $c<c_1$ (respectively $c>c_2$). In fact, the related histograms are nicely overlapped, hence corroborating the scalings for $R(c, k)$ versus $n$.}\label{Distr}
\end{figure}


\subsection{The case $k>1$}\label{ssec:K2}

Before proceeding with the analysis of  the case $k>1$, it is worth recalling that, the higher $k$, and the more connected the graph $G_n$. In fact, given the building rule for $\mathcal{G}(n,k,p)$, any couple of nodes $i$ and $j$ with $|i-j|\leq k$ always turns out to be connected and the number of links deterministically inserted scales as $nk$. Thus, the higher $k$, and the smaller the diameter (i.e., the length of the longest among the all possible shortest paths). 
This remark has important effects on the spreading process since a small diameter implies that any node is reachable in a small number of steps.

Now, we run the Maki-Thompson dynamic described by the rules (\ref{eq:transitions}) on a set of graphs $G_n$, with $k=2,3,4$, respectively.
Performing the same dynamic and average operations adopted for $k=1$, 
we get our estimate for $R(c,k)$ which is shown in Fig.~\ref{k1234bin}. 
\newline
First, we notice that curve collapses analogous to those highlighted for $k=1$ (see Fig.~\ref{collapse1}) hold and the emerging scalings are still $R(c,k) \sim O(1)$ for $c<c_1(k)$ and $R(c,k) \sim n$ for $c>c_2(k)$. In order to further corroborate these scalings, in analogy with the case $k=1$, we checked that the distributions $p_{i}(r;k,c)$, for $i=1,2$, are consistent with the above picture.
\newline
Moreover, our estimates for the thresholds evidence that $c_1$ and $c_2$ both decrease with $k$. This is consistent with expectations since, as $k$ increases, the network gets more connected and, even with a relatively small number of short-cuts (i.e., ultimately, a small $c$), the rumour can spread easily. On the other hand, it is easier escaping from the minor propagation regime.
Our estimates for $c_1$ and $c_2$ allow us to build a phase diagram in the space $(c,k)$ which highlights the different regimes of propagation (see Fig.~\ref{diagramma}, left panel).

\begin{figure}[tb]
	\centering
	\includegraphics[width=12.6cm]{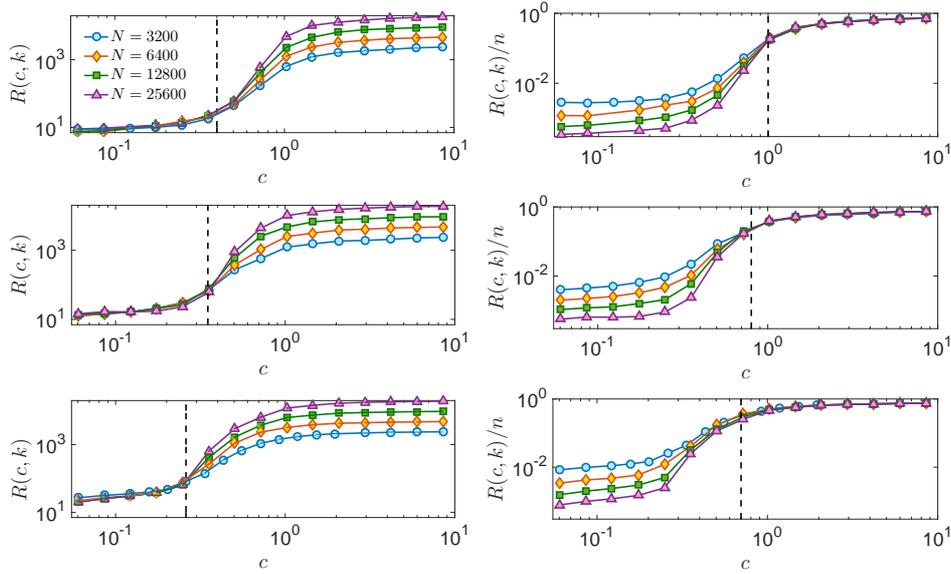}
	\caption{The average number $R(c,k)$ of removed individuals (left panel) and the normalized average number $R(c,k)/n$ of removed individuals (right panel) are shown versus $c$, for different values of $n$, as explained by the common legend in the first panel. Several values of $k$ are also considered: $k=2$ (first row), $k=3$ (second row), and $k=4$ (third row).  The vertical dashed lines indicate the thresholds $c_1$ (left panels) and $c_2$ (right panels) below which and above which, respectively, the collapse occurs. As one can see, the larger $k$ and the smaller the critical thresholds. Notice the logarithmic scale in all panels. }\label{k1234bin}
	\end{figure}

\begin{figure}[tb]
	\centering
	\includegraphics[width=11cm]{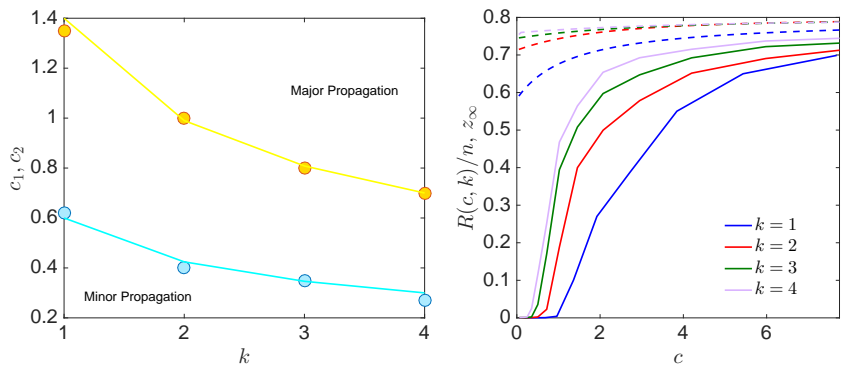}
	\caption{Left panel: the estimates of $c_1(k)$ and of $c_2(k)$ extracted from Figs.~\ref{collapse1} and \ref{k1234bin} are shown versus $k$. The resulting phase diagram highlights the region in the $(c,k)$ plane, where the spreading dynamics is expected to end up with a major propagation or with a minor propagation. There also exists an intermediate region corresponding to $c_1(k) \leq c \leq c_2(k)$, where the scaling of $R(c,k)$ with $n$ can not be described by $n^{\gamma}$, with $\gamma=0,1$. Right panel: Normalized average number $R(c,k)/n$ of removed individuals versus $c$ for different values of $k$, as explained by the legend. Results from simulations (solid lines) are compared to the numerical solution of Eq.~\ref{eq:MF} (dashed lines) obtained following a mean-field approach. Notice that the latter provides an upper bound and the approximation improves as $c$ gets larger.}\label{diagramma}
\end{figure}

Finally, we compare the estimate for $z_{\infty}$, obtained from the mean-field approximation, with the numerical results (see Fig.~\ref{diagramma}, right panel). One can see that $z_{\infty}$ provides an upper bound which is rather poor in the region of small $c$ (as expected because there the network still exhibits a inhomogneous topology), while it gets better for larger $c$ and larger $k$ (as expected because the network approaches the fully connected graph). Also, the upper bound depends less sensitively on $k$ and on $c$, and it saturates to $0.8$ consistently with the analytical results collected in \cite{sudbury-JAP1985}.

To summarize, the rumour is likely to diffuse easily when the network is strongly connected, that is, when $k$ is large a relatively small value of $c$ is required in order for the rumour to propagate.

\section{Conclusions}
Spreading phenomena are widespread and play a significant role in a variety of real-life systems. 
In fact, the ``object'' that is spreading can be a piece of information, a belief, an innovation, a fad, an  infection, a malware, a practice, etc. 
The ``space'' where the spreading takes place can be a society, a community, the Internet, etc. and in most cases it can be described in terms of networks.
According to the context, the spreading may be desirable or undesirable, and, in any case, being able to predict its evolution can be crucial for preventing the outburst of epidemics, to shape public opinion, to impact on financial markets, or to ensure cyber-security.

In the last decades, the progressive globalisation and the awareness of the relative opportunities and dangers has prompted intensive investigations in this field of research and many spreading models have been introduced. An important outcome that was highlighted concerns the strong influence of the topology of the underlying network on the spreading evolution (see e.g., \cite{Dorogovstev-Rev2002,Newman-2010,Barrat-2012}).

In this work we considered the Maki-Thompson (MT) model for spreading and, to fix ideas, we shall refer to a rumour (as the object which is spreading) propagating within a society (as the space where the object is spreading). The MT model is based on the following rules: the population is divided into ignorants, spreaders, and stiflers. Any spreader tries to pass the rumour to a neighbours of hers, chosen randomly. If the latter is ignorant, then it becomes a spreader, otherwise the initial spreader becomes a stifler. In a finite population the process will eventually reach a stack situation where individuals are either stiflers or ignorants and we are interested in this final configuration. A final, large number of stiflers means that the rumour has succeeded in reaching a vast fraction of the population; vice versa, a final, large number of ignorants means that a vast fraction of the population has remained unaware of the rumour.    
\newline
We embedded the model on a small-world  (i.e., highly clustered) network, where the average coordination number and the diameter are controlled by a tuneable parameter $c$. More precisely, starting from a $k$-regular ring lattice, we inserted randomly short-cuts in such a way that, in the average, the degree of a node is $k+c$.
\newline
We studied the evolution of the MT spreading process on such a class of networks exploiting technics of coupling between the original process and suitable defined growth processes which are related to the well known Galton-Watson process. Then we used results on the phase transition of survival and extinction of a branching process to obtain bounds for the number of stiflers at the end of the process.
In this way we can prove the existence of a phase transition regarding the extent of the propagation: for relatively large values of $c$ the rumour is propagated to at least $O(n^{\gamma})$ individuals with asymptotic non-zero probability; for relatively small values of $c$ the rumour is propagated to at most $\log n$ individuals with high probability. We denote with $c_1$ and $c_2$ the thresholds distinguishing low values ($c<c_1$) and high values ($c>c_2$) of $c$; both threshold values are finite and depend on $k$.
\newline
This picture is confirmed by numerical simulations which also allowed us to get an estimate for the critical values $c_1(k), c_2(k)$ and for $\gamma$. In fact, we recovered the existence of two regimes characterized by different scalings: when $c<c_1$ the final number of stiflers remains finite and independent of the system size (i.e., $\gamma =0$), while when $c>c_2$ the final number of stiflers scales linearly withe the system size (i.e., $\gamma =1$). 
\newline
The model was also approached via mean-field equations.
In this way we can obtain an estimate for the final value of stiflers, $z_{\infty}:=\lim_{t\rightarrow\infty}z(t)$, where $z(t)$ is the solution of a system of mean field equations, in a time rescaled process of the Maki-Thompson rumour spreading process. 
This estimate is again corroborated via numerical simulations showing that $z_{\infty}$ constitutes an upper bound for the final number of stiflers and the bound is better (i.e., closer to the result obtained via simulations) as $c$ and $k$ gets larger. This is consistent with the assumptions underlying the mean-field approach. 

Our work has highlighted that small changes in the topology of the networks can give rise to abrupt changes in the emerging phenomenology. If the network is highly clustered (i.e., a relatively large number of short-cuts), then a major propagation occurs. Otherwise, the rumour is mainly locally propagated, through paths of vertices connected by local-edges.
These conclusions also open the way to further analysis meant to clarify the robustness of the result in the presence of weights on the links, or in the presence of noise which impairs the propagation of the rumour.



\section{Appendix A - Proof of Theorem \ref{teor:phaset}}

Our proofs are constructive and rely mainly in a comparison (coupling) between the Maki-Thompson rumour process and suitable defined branching processes. Then we apply well known results regarding survival and extinction of these processes to obtain conditions for major and minor outbreak in the rumour propagation. We refer the reader to \cite{BranchingProcesses} for a review of the theory of branching processes. 

\subsection{Subcritical regime}

In order to prove Theorem \ref{teor:phaset}(i) we construct a subcritical branching process  whose total progeny dominates the final number of informed individuals in the rumour process. Roughly speaking, there are two ways of spreading the rumour from a given vertex, namely, locally, i.e., through paths of vertices connected by local edges, that we call local clusters, or through shortcuts to other vertices on the graph. We will see that, for $c$ small enough, the size of each new local cluster may be bounded from above for a value that does not depend on $n$, and also, that the growth of new transmissions through shortcuts is related to the growth of the suitably defined subcritical branching process.      

\subsubsection{Local and blocked clusters}

In what follows, consider the Maki-Thompson rumour spreading  process on $\GG (n,k,p)$, with $p=c/(n-2k+1)$, $c$ a constant independent of $n$ and $k\geq1$. Before to state the first definitions, we give an order to the vertices of the set $\mathcal{R}(\tau_n)$, i.e., we write $\mathcal{R}(\tau_n)=\{v_0,v_{1},\dots,v_{r-1}\}$, where $r:=R(\tau_n) $ and $v_{j}$, $j=1,\dots,r$, is the $j$-th informed vertex on $[n]$ up to time $\tau_n$. 

\smallskip
\begin{defn}[\textit{The predecessor}]
We define the predecessor of a vertex $v_{j}\in  \RR(\tau_n)$, denoted by $pred(v_{j})$, as the vertex $u\in \mathcal{N}(v_{j})$ such that the event
$$
\cup_{t=1}^{\tau_n}\big(\{v_{j}\in \II(t)\}\cap \{u\in \SS(t)\}\cap \{U_t=u\}\cap \{U_u^m=v_{j}\}\big),
$$
occurs. If such a vertex does not exist, we say that vertex $v_j$ has no predecessor.
\end{defn}
In words, $pred(v_{j})$ is the first spreader vertex that informs $v_{j}$. Observe that there is not a predecessor for $v_0$, since this vertex is already a spreader at time $t=0$. 

We are interested in identifying the set of those vertices influenced for a given vertex, in the sense of the transmission of the rumour, but only through a sequence of local edges. Such a set, that we call as local cluster, determines the local range of spreading  from a spreader vertex on the graph.     

\smallskip
\begin{defn}[\textit{Local cluster}]\label{def:loccluster}
Given two vertices $v_i,v_j \in \RR(\tau_n)$, $i<j$, we say that $v_j$ is informed from $v_i$ through local edges if, there exist a sequence of vertices $v_i=u_1,u_2,\ldots,u_{\ell}=v_j$ such that $u_{i}\in \mathcal{N}_{\ell}(u_{i-1})$ and $pred(u_i)=u_{i-1}$, for all $i=1,2,\ldots,\ell$. We denote this event as $v_i\rightarrow v_j$, and we define the local cluster for a given vertex $v_i\in \RR(\tau_n)$,  as the set
$$\mathcal{C}_{v_i}(\tau_n):=\{v\in \RR(\tau_n): v_i\rightarrow v\}.$$ 
\end{defn}

Our  purpose  now, is to bound in a certain way the size of a local cluster. Before that, we need some additional definitions. 

\smallskip
\begin{defn}[\textit{Blocking vertices}]
For any vertex $v\in [n]$ we say that

\begin{enumerate}
\item $v$ is a right blocking vertex ({\it rbv}) if $U_v^1 = v-1$;
\item $v$ is a left blocking vertex ({\it lbv}) if $U_v^1 = v+1$.
\end{enumerate}
\end{defn}

\smallskip
\begin{defn}[\textit{Blocked cluster}]\label{defn:blockedc}
For any vertex $v\in [n]$ we define its blocked cluster as the random interval
\begin{equation}\label{eq:Bv}
B_v := v +  \llbracket -J_v^{-}, J_{v}^{+}  \rrbracket ,
\end{equation}
where
$$J_v^{-} := \min\left\{i>0: \bigcap_{j=0}^{k-1} \{v-i+j \text{ is a }lbv\}\right\},$$
and
$$J_v^{+} := \min\left\{i>0: \bigcap_{j=0}^{k-1} \{v+i-j \text{ is a }rbv\}\right\},$$
See Fig. \ref{FIG:defbc}
\end{defn}

\begin{figure}[h]
\label{FIG:defbc}
\begin{center}
\begin{tikzpicture}

\draw[snake=brace,segment aspect=0.5] (-1.2,0.4) -- (5.2,0.4);
\draw (2,1.2) node[below,font=\footnotesize] {$B_v$};

\draw (2.5,-0.1) node[below,font=\footnotesize] {$v$};
\draw (8,0) node {...};
\draw (-4,0) node {...};
\filldraw [black] (-3.5,0) circle (2pt);
\filldraw [black] (-3,0) circle (2pt);
\draw (-2.5,0) node {...};
\filldraw [black] (-2,0) circle (2pt);
\draw (-1.5,0) node[font=\footnotesize] {$\oplus$};
\draw (-1.25,0) node[font=\footnotesize] {$\big($};
\draw (-1,0) node[font=\footnotesize] {$\ominus$};
\draw (-0.5,0) node[font=\footnotesize] {$\ominus$};
\filldraw [black] (0,0) circle (2pt);
\filldraw [black] (0.5,0) circle (2pt);
\draw (1,0) node[font=\footnotesize] {$\oplus$};
\filldraw [black] (1.5,0) circle (2pt);
\filldraw [black] (2,0) circle (2pt);
\filldraw [black] (2.5,0) circle (2pt);
\draw (3,0) node[font=\footnotesize] {$\oplus$};
\filldraw [black] (3.5,0) circle (2pt);
\filldraw [black] (4,0) circle (2pt);
\draw (4.5,0) node[font=\footnotesize] {$\oplus$};
\draw (5,0) node[font=\footnotesize] {$\oplus$};
\draw (5.25,0) node[font=\footnotesize] {$\big)$};
\filldraw [black] (5.5,0) circle (2pt);
\filldraw [black] (6,0) circle (2pt);
\draw (6.5,0) node {...};
\filldraw [black] (7,0) circle (2pt);
\draw (7.5,0) node[font=\footnotesize] {$\ominus$};
\end{tikzpicture}
\end{center}
\caption{Representation of a blocked cluster for $k=2$. The set o vertices of $G_n$ may be partitioned in those vertices which are $lbv$ ({\footnotesize $\ominus$}), in those which are $rbv$ ({\footnotesize $\oplus$}), and others ($\bullet$). 
For simplicity we have not drawn the incident edges.}
\end{figure}
\smallskip
\begin{prop}\label{prop:maxsizeJv}
Let $x(n)=O(n^{\beta})$ and $\beta<1/2$ be a constant. Then $J_v^+, J_{v}^{-} < x(n)$ {\it w.h.p.}  
\end{prop}
\begin{proof}
For any vertex $v\in [n]$, let $A_v$ be the event that $v$ is not a $lbv$, and $q_v=\bbP(A_v)$. By independence of the random variables $U_v^1$,
$$
\bbP(J_v^-\geq x(n))=\bbP\left(\cap_{j=k}^{x(n)-1}A_{v-x(n)+j}\right)\leq q_{v_0}^{x(n)-k},
$$
where 
$$
q_{v_0}=\sum_{j=0}^{n-2k-1}\left(1-\frac{1}{2k+j}\right)\bbP(|\NN_{s}(v_0)|=j),
$$
and $|\NN_{s}(v_0)|$ follows a Binomial distribution with parameters $(n-2k-1,p)$. Using Theorem 6.1 in \cite{durrettbook}, we get that $q_{v_0}\rightarrow 1-\bbE(\{X+2k\}^{-1}):=q$, as $n\rightarrow\infty$, where $X$ follows a Poisson distribution with mean $c$. Therefore, since $q<1$ and $x(n)\rightarrow\infty$ as $n\rightarrow\infty$, then we have that $\bbP(J_v^-\geq x(n))\rightarrow0$ as $n\rightarrow\infty$.
The proof for $J_v^+$ follows the same lines as the $J_v^-$ and we ommit it here.
\end{proof}

\begin{prop}\label{prop:convJv}
For any $v\in [n]$ fixed, we have that
$$J_v^{-}, J_v^{+} \stackrel{\mathcal{L}}{\rightarrow} X_k,$$
as $n\rightarrow \infty$, where $\stackrel{\mathcal{L}}{\rightarrow}$ denotes convergence in law, and $X_k$ is the number of coin flips up to get $k$ consecutive heads, assuming a coin with probability $\alpha$ of heads, given by
\begin{equation}\label{alpha}
\alpha:=\alpha(k,c)= \bbE\left(\{X+2k\}^{-1}\right),
\end{equation}
where  $X$ follows a Poisson distribution, $Poisson(c)$.
\end{prop}

\begin{proof}
Let $v\in[n]$. We proof the result for $J_v^{+}$. The proof of the corresponding result for $J_{v}^{-}$ is the same, and therefore omitted here. In order to study the distribution of $J_v^{+}$, we have to deal with the randomness coming from the shortcuts incidents in $I_v:=\llbracket v+1,\ldots v+x(n) \rrbracket$. First, we note that {\it w.h.p.} no pair of vertices inside this interval are connected through a shortcut. Indeed, for any $u\in I_v$ we have
\begin{equation*}
\bbP\left(I_v \cap \NN_s (u)\neq \emptyset\right) \leq 1-\left(1-\frac{c}{n-2k-1}\right)^{x(n)}\sim 1- e^{-c\,x(n)/n},
\end{equation*}
and therefore
\begin{equation}\label{eq:noshortcuts}
\bbP\left(\cup_{u\in I_{v}} \left\{I_v \cap \NN_s (u)\neq \emptyset \right\}\right)\leq x(n)\left[1-\left(1-\frac{c}{n-2k-1}\right)^{x(n)} \right]\sim x(n)\left(1- e^{-c\,x(n)/n}\right)=o(1),
\end{equation}
where the $o(1)$ term is due to $x(n)=O(n^{\beta})$, and $\beta <1/2$ is a constant. 

By Proposition \ref{prop:maxsizeJv} we know that $\bbP(J_v^-\geq x(n))=o(1)$. Let $\ell < x(n)$ be fixed, and denote as $A_{nsc}$ the event that there is not shortcuts between a pair of vertices in $I_v$. Conditioning on $A_{nsc}$, we obtain 
\begin{equation}\label{eq:expresJv}
\bbP(J_{v}^{-}=\ell)=o(1)+\bbP(X_{k}^{x(n)}=\ell),
\end{equation}
where $X_k^{{x(n)}}$ is the number of coin flips up to get $k$ consecutive heads, in $x(n)$ flips, assuming a coin with probability $\bbP(U_u^1=u-1|A_{nsc})$ of coming up heads, for $u\in I_v$. The $o(1)$ term in Eq. \eqref{eq:expresJv} comes from \eqref{eq:noshortcuts}, which implies $\bbP(A_{nsc})=1+o(1)$. On the other hand, the comparison with successive coin flips is possible because, on $A_{nsc}$, a vertex is a $rbv$ independently of other vertices in $I_v$. Now, observe that 

\smallskip
\begin{equation}\label{eq:unifAnsc}
\begin{array}{rcl}
\bbP\left(U_{u}^1=u-1| A_{nsc}\right) &= &\sum_{j=0}^{n-x(n)}\bbP(U_{u}^1=u-1| |\NN_{s}(u)|=j, A_{nsc})\,\bbP(|\NN_{s}(u)|=j|A_{nsc})\\[.3cm]
                 & = &  \sum_{j=0}^{n-x(n)}\left(\tfrac{1}{2k+j}\right)\,\bbP(|\NN_{s}(u)|=j| A_{nsc})\\[.3cm]
                 &=&(1+o(1))\bbE\left(\{X+2k\}^{-1}\right),
\end{array}
\end{equation}
where $X$ follows a Poisson distribution, $Poisson(c)$ (see \cite{durrettbook} Theorem 6.1), and the last equality is obtained by noting that, conditioned on $A_{nsc}$, the random variable $|\NN_{s}(u)| \sim Binomial(n-x(n),c/(n-2k-1))$, and $x(n)=o(n)$. Therefore, we have from \eqref{eq:expresJv} and \eqref{eq:unifAnsc} that

$$
\bbP(J_{v}^{-}=\ell)=o(1)+\bbP(X_{k}=\ell),
$$
where $X_k$ is the number of coin flips up to get $k$ consecutive heads assuming a coin with probability $\alpha$ of heads, and $\alpha= \bbE\left(\{X+2k\}^{-1}\right)$.

\end{proof}

\smallskip
\begin{cor}\label{cor:B0}
For any $v\in [n]$ fixed, let $B_v$ be its blocked cluster. Then,
$$\bbE\left(|B_v|\right)= \left[1+ 2\,  \left(\frac{\alpha^{-k}-1}{1-\alpha}\right)\right] +   o(1),$$
where $\alpha$ is given by \eqref{alpha}.
\end{cor}

\begin{proof}
By Definition \ref{defn:blockedc} and Proposition \ref{prop:convJv} we have that 
$$\bbE\left(|B_v|\right)=1+2\, \bbE\left(J_v^{-}\right)=1+2\, (1+o(1)) \, \bbE\left(X_k\right),$$
where $X_k$ is the number of coin flips up to get $k$ consecutive heads, assuming a coin with probability $\alpha$ given by \eqref{alpha}. In addition, it is well know that
$$\bbE\left(X_k\right)=\frac{\alpha^{-k}-1}{1-\alpha},$$
which completes the proof. To obtain the last expression, observe that one may apply successively the first step analysis (see for example \cite[p. 116]{karlin/taylor/1998}), conditioning on the result of the first coin flip, to obtain

$$\bbE\left(X_k\right) = \frac{1+\alpha + \alpha^2 + \cdots + \alpha^{k-1}}{\alpha^k}=\frac{1-\alpha^k}{\alpha^k(1-\alpha)}.$$
\end{proof}

\smallskip
\begin{lem}\label{lem:local-blocked-cluster}
For any $v \in \RR(\tau_n)$, $\CC_v(\tau_n) \subseteq B_v$ {\it w.h.p.}
\end{lem}

\begin{proof}
Let $v\in \RR(\tau_n)$, and assume that there exist $u\in\CC_v(\tau_n)$ such that $u\not\in B_v$.  By definition of local cluster, there exist a sequence of vertices $v=u_1,u_2,\ldots,u_{\ell}=u$ such that $u_{i}\in \mathcal{N}_{\ell}(u_{i-1})$ and $pred(u_i)=u_{i-1}$, for all $i=1,2,\ldots,\ell$. As $u\not\in B_v$, we have by definition of blocked cluster that at least one of these vertices, say $u_i$, for $i<\ell$, is a {\it lbv} or a {\it rbv}, and this in turn implies $pred(u_{i+1})\neq u_i$. We get a contradiction.  
\end{proof}

\subsubsection{Emergence of local clusters and a subcritical branching process}
\label{SS:emergence}

The idea behind the proof of Theorem \ref{teor:phaset} (i) is to compare the emergence and growth of local clusters with a subcritical branching process. This comparison allows to dominate the final number of informed individuals in the rumour process by the corresponding set of individuals in the branching process.   
We will see that those vertices which are informed through shortcuts connections will have an important role in the spreading process. We start by labelling these vertices as $v_{I_1},v_{I_2},\ldots,v_{I_{\ell}}$. That is, $v_{I_h}$ is the $h$-th vertex in $\RR(\tau_n)$, such that $pred(v_{I_h})\in \NN_{s}(v_{I_h})$. Observe that $\ell< r$. Also, for the sake of simplicity, we denote $\CC_0:=\CC_{v_0}(\tau_n)$, $B_0:=B_{v_0}$, and $\CC_{j}:=\CC_{v_{I_j}}(\tau_n)$, $B_{j}:=B_{v_{I_j}}(\tau_n)$ for all $j=1,\ldots,\ell$. The key ingredients for the proof are the following claims, which result as consequence of the construction of the process and Lemma \ref{lem:local-blocked-cluster}.

\smallskip
\begin{enumerate}
\item[]Claim 1. $\RR(\tau_n) = \cup_{j=0}^{\ell} \CC_{j}$.\\
\item[]Claim 2. $\cup_{j=0}^{\ell} \CC_{j}  \subseteq \cup_{j=0}^{\ell}B_j$ {\it w.h.p.}\\ 
\end{enumerate}

\smallskip
Claim 1 and 2 imply that by dominating the number of vertices in the respective blocked vertices we can dominate the number of removed vertices at the end of the process. In order to control the emergence of local clusters from above we use an exploration process of blocked clusters. We define the sets $\UU_0:=[n]\setminus B_0, \AA_0:=B_0, \RR_0:=\emptyset$, and for  $h\geq 0$ let
\begin{equation}\label{eq:exploration}
\begin{array}{rcl}
\RR_{h+1}&=&\RR_h \cup \AA_h\\[.2cm]
\AA_{h+1} &=& \left\{v\in \UU_{h}:\cup_{u\in \NN_s(\AA_{h})} [v\in B_{u}]  \right\}\\[.2cm]
\UU_{h+1}&=& \UU_h - \AA_{h+1},
\end{array}
\end{equation}
\smallskip
\noindent
where, for any $A\subset [n]$ we denote by $\NN_s(A)$ the set of vertices connected to a vertex in $A$ through a shortcut, i.e.,
\begin{equation}\label{eq:neiSC}
\NN_s(A):=\{v\in [n]:\cup_{u\in A}\left[I_{uv}=1\right]\}.
\end{equation}
\smallskip
\noindent
Observe that $\RR(\tau_n)\subseteq \cup_{h=0}^{\infty} \AA_h$. For $h=0,1,\ldots$, let $Y_{h}: = |\AA_{h}|$.

\smallskip
 \begin{lem}\label{lema:BPcomp}
The process $(Y_h)_{h\geq 0}$ is dominated by a branching process $(Z_h)_{h\geq 0}$ with mean offspring distribution equal to
\begin{equation}\label{eq:meansub-branch}
m_n(k,c):=\left(1+ 2\, \left\{\frac{\alpha^{-k}+1}{1-\alpha}\right\}\right)\,c\, +\, o(1),
\end{equation}
where $\alpha$ is given by \eqref{alpha}. 
 \end{lem}
 
 \begin{proof}
We introduce two new independent collections of random variables, the set $\{X_{uv}^h:u,h\geq 1, v\in[n]\}$ of i.i.d. Bernoulli random variables of parameter $p$, and also the set $\{\tilde{X}_u^h :u,h\geq 1\}$ of i.i.d. random variables distributed as $|B_0|$. We define the branching process $(Z_{h})_{h\geq 0}$ as follows. Let $Z_0=|B_0|$, and for $h\geq 0$ define
\begin{equation}\label{eq:subBP}
Z_{h+1}=\sum_{u\in \AA_h, v\in \UU_h} \tilde{X}_v^{h+1} I_{u,v} + \sum_{u\in \AA_h, v\in \UU_h^{c}} \tilde{X}_v^{h+1} X_{u,v}^{h+1} + \sum_{u=n+1}^{n+Z_h - |\AA_h|} \sum_{v=n+1}^{2n} \tilde{X}_v^{h+1} X_{u,v}^h.
\end{equation} 

The idea is to compare individuals in the branching process with vertices of the exploration process of blocked clusters. Thus, the second term in \eqref{eq:subBP} represents the set of additional births in the branching process due to $|\UU_h|<n$, and the third term is the set of children of those individuals in the branching process that are not vertices in $\AA_h$. On the other hand, observe from the first term that
$$\sum_{u\in \AA_h, v\in \UU_h} \tilde{X}_v^{h+1}\, I_{u,v} = \sum_{u\in \AA_h, v\in \UU_h} |B_v|\, I_{u,v} \,\geq\, |\AA_{h+1}|.$$
It follows from the construction that the process $(Z_{h})_{h\geq 0}$ is a branching process with mean offspring distribution given by $m_n(k,c):=\bbE(|B_0|)\, c$, and $Z_h \geq Y_h$. Expression \eqref{eq:meansub-branch} for $m_n(k,c)$ is obtained by Corollary \ref{cor:B0}.
 \end{proof}
 
 \subsubsection{Proof of Theorem \ref{teor:phaset} (i)} 
 
All the random variables, and processes defined in the previous sections depend of $n$, and we have suppressed this dependence for the sake of simplicity. Now, we want to compare our results for different values of $n$, and therefore we include $n$ in the notation. By \eqref{eq:exploration} we have that 
$\RR(\tau_n)\subset \cup_{h=0}^{\infty}\AA_h^n$, which implies 
$$R(\tau_n)\,=\,|\RR(\tau_n)|\,\leq\, \sum_{h=0}^{\infty}Y_h^n,$$
where $Y_h^n=|\AA_{h}^n|$. Then, by Lemma \ref{lema:BPcomp} we obtain
$$R(\tau_n) \, \leq\,  \sum_{h=0}^{\infty} Z_{h}^{n} =:Z^n,$$
where $Z^n$ denotes the total progeny of the branching process $(Z_{h}^n)_{h\geq 0}$, with offspring distribution given by
$$
\sum_{i=1}^{X^n} \tilde{X}_i^n,
$$
where $X^n$ is a Binomial random variable, $Binomial\left(n-2k-1,c/(n-2k-1)\right)$, and $\tilde{X}_1^n, \tilde{X}_2^n, \ldots$ are a sequence of i.i.d. with the same distribution as $|B_0^n|$. By Definition \ref{defn:blockedc}, and Proposition \ref{prop:convJv} we have that $|B_0^n|$ converges in law, as $n\rightarrow \infty$, to $1+X_{-}^{k}+X_{+}^k$, where $X_{-}^{k}$, $X_{+}^{k}$ are independent copies of the number of coin flips up to get $k$ consecutive heads assuming a coin with probability $\alpha$ of heads, with $\alpha= \bbE\left(\{X+2k\}^{-1}\right)$, and $X$ is a Poisson random variable  $Poisson( c )$. This in turns, imply
\begin{equation}\label{eq:limit}
\sum_{i=1}^{X^n} \tilde{X}_i^n \, \stackrel{\mathcal{L}}{\longrightarrow}\,  \sum_{i=1}^{X} \tilde{X}_i^{\infty},
\end{equation}
as $n\rightarrow \infty$, where $\tilde{X}_i^{\infty}  \stackrel{\mathcal{L}}{=}  1+X_{-}^{k}+X_{+}^k$, for all $i$. In other words, we have a sequence of branching processes whose offspring distribution converge to the limit distribution defined by \eqref{eq:limit}. Therefore, the total progeny of the process  $(Z_{h}^n)_{h\geq 0}$, $Z^n$, converges in law, as $n\rightarrow \infty$, to the total progeny, $Z$, of a branching process with mean given by
$$m(k,c):=\left(1+ 2\, \left\{\frac{\alpha^{-k}+1}{1-\alpha}\right\}\right)\,c.$$
An explicit expression for the distribution of the total progeny in a branching process may be found in \cite{dwass}. Finally, we can choose $c$ sufficient small in such a way that $m(k,c)<1$. This can be justified by noting, after some calculations, that as a function of $c$, $m(k,c)$ is right-continuous at $[0,\infty)$, non-decreasing, and $m(k,0)=0$. Then the respective branching process is subcritical, and $\bbP\left(Z < \infty\right)=1$. We complete the proof by noting that, for $k=1,2,\ldots$
$$\liminf_{n\to \infty}\bbP\left(R(\tau_n)\leq \ln n\right)\geq  \liminf_{n\to\infty}\bbP\left(Z^n\leq \ln n\right)  \geq  \liminf_{n\to\infty}\bbP\left(Z^n\leq k\right) = \bbP\left(Z\leq k\right),$$
and letting $k\to \infty$, yields
$$\liminf_{n\to \infty}\bbP\left(R(\tau_n)\leq \ln n\right) \geq  \lim_{k\to \infty} \bbP\left(Z\leq k\right) = \bbP\left(Z < \infty\right)=1.$$


\subsection{Supercritical regime}

In this section we shall prove Theorem  \ref{teor:phaset}(ii), which states that with positive probability there is a constant $c_2(k)\in(0,\infty)$, such that the total number of removed vertices at the end of the evolution of the rumour spreading process is at least of order  $O(n^{\gamma})$, $0<\gamma\leq1$, provided $c>c_2(k)$. 
The idea in brief is to analyze the number of those removed vertices at time $\tau_n$ containing $v_0$, who were informed through a shortcut in a restricted version of the process. This restricted version is defined to consider that each spreader attempts to transmit the information \textit{at most twice}. We will prove that the number of removed vertices in the restricted process, which is smaller than the total quantity of removed vertices at time $\tau_n$ in the original process, is related to the size of a  supercritical Galton-Watson branching process provided $c$ is sufficiently large.

Let $\kappa_{v_0}(\tau_n)$ denotes the cluster of removed vertices at time $\tau_n$ containing $v_0$  in the restricted process. We reveal $\kappa_{v_0}(\tau_n)$ as follows. Define the sets  $\VV_{0}:=[n]\setminus \{v_0\}$, $\YY_{0}:=\{v_0\}$, and for $r\geq 0$ let

\begin{equation}\label{eq:explorationBFS}
\begin{array}{rcl}
\YY_{r+1} &=& \left\{U_{v}^i: U_{v}^i \in \NN_s(v), i=1,2, v\in \YY_r \right\}\\[.2cm]
\VV_{r+1}&=& \VV_r - \YY_{r+1}.
\end{array}
\end{equation}

Then observe that
$$\kappa_{v_0}(\tau_n):=\cup_{r=0}^{\infty}\YY_{r} \subseteq \RR(\tau_n).$$
Our aim now is to study $\kappa_{v_0}(\tau_n)$. 
As in subsection \ref{SS:emergence}, we will look at the vertices which are informed through shortcuts connections, that is, at the set of vertices  $\{v_{I_h}\}_{h\geq1}^{\ell}$, $\ell< R(\tau_n)$, where as before $v_{I_h}$ denotes the $h$-th vertex in $\RR(\tau_n)$, such that, $pred(v_{I_h})\in\mathcal{N}_s(v_{I_h})$. Let $v_{I_0}:=v_0$.  

Next we introduce  some definitions and lemmas necessary to prove Theorem  \ref{teor:phaset}(ii). 

\begin{defn}[truncated local cluster]\label{def:truncluster}
Let $N:=N(n)\in\bbN$ with $N\leq R(\tau_n)$. We define for each vertex  $v_{I_h}$, $h=0,\dots,N-1$, its truncated local cluster $(T\ell c)$ as the set
$$\mathcal{C}_{v_{I_h}}^T(t_N):=\Big\{u\in \{\RR(t_N)\cup \SS(t_N)\}: v_{I_h}\rightarrow u\Big\},$$
where $t_N$ is the time at which $v_{I_{N}}$ is informed, i.e.,
$$t_N:=\inf\{t\geq 0: v_{I_{N}}\in \mathcal{S}(t)\}.$$ 
\end{defn}

Observe that at time $t_N$ we have $N-1$ vertices informed through shortcuts, since $v_{I_0}$ is informed by construction. Moreover, note that the vertices inside $\mathcal{C}_{v_{I_h}}^T(t_N)$, $h=0,\dots,N-1$, are connected through local edges, and all are either removed or spreaders at time $t_N$. 

Now for each $h=0,\dots,N-1$, let $C_{v_{I_h}}$ be the set of all vertices connected by local edges in $\mathcal{C}_{v_{I_h}}^T(t_N)$. Note that $C_{v_{I_h}}$ may contain ignorants vertices at time $t_N$. Moreover, these sets of vertices $C_{v_{I_h}}, h=0,\dots,N-1$, are not necessarily disjoint. For the next two lemmas we will take the following disjoint sets of vertices  constructed from  $\{C_{v_{I_h}}\}_{h=0}^{N-1}$ and the bloqued clusters $\{B_{v_{I_h}}\}_{h=0}^{N-1}$ of Definition  \ref{defn:blockedc}. Let
$
\mathfrak{C}_{v_{I_0}}^T(t_N):=C_{v_{I_0}}$,  $\mathfrak{B}_{v_{I_0}}:=B_{v_{I_0}}$ and  for  $h=1,\dots, N-1$, let
\begin{equation}\label{DisjBC}
\mathfrak{C}_{v_{I_h}}^T(t_N):=C_{v_{I_h}}\setminus\Big(\cup_{j=0}^{h-1}C_{v_{I_j}}\Big) \quad\text{and}\quad  \mathfrak{B}_{v_{I_h}}:=B_{v_{I_h}}\setminus\Big(\cup_{j=0}^{h-1}B_{v_{I_j}}\Big).
\end{equation}

\begin{defn}
We call a path of  $m$ vertices in $\{\mathfrak{B}_{v_{I_h}}\}_{h=0}^{N-1}$, $2\leq m\leq N$, to a sequence of vertices such that: two consecutive vertices in the sequence are connected by a local edge if they are inside of the same blocked cluster, say $\mathfrak{B}_{v_{I_h}}$, and connected by a shortcut if they are in different blocked clusters, say $\mathfrak{B}_{v_{I_h}}$ and $\mathfrak{B}_{v_{I_h'}}, h\neq h'$. In this definition no repetitions of vertices and edges allowed other than the repetition of the starting and ending vertex, and in such situation we say that the path of vertices is a closed path. 
\end{defn} 

\begin{defn}
We say that $m$ disjoint blocked clusters in $\{\mathfrak{B}_{v_{I_h}}\}_{h=0}^{N-1}$ form a cycle, if there exist at least one closed path between them.  
\end{defn}

Recall that the vertices inside the disjoint blocked clusters are connected through local edges in $\mathcal{G}(n,k,p)$. Therefore, the previous definition tell us that $m$ disjoint blocked clusters form a cycle, $2\leq m\leq N$, if for $m=2$ there exist at least $2$ shortcuts between the $2$ blocked clusters, while for $m>2$ there are at least $m$ shortcuts between the $m$ blocked clusters, each joint two different blocked clusters.   

Let $X_N$ denotes the random variable which counts the total number of cycles between the disjoint blocked clusters $\{\mathfrak{B}_{v_{I_h}}\}_{h=0}^{N-1}$. That is, the sum over $m$, for $m=2,3,\dots,N$, of cycles between $m$ disjoint blocked clusters in $\{\mathfrak{B}_{v_{I_h}}\}_{h=0}^{N-1}$. 

\begin{lem}\label{existenceTv0}
Let $0<\beta<1/2$, $\gamma:=\gamma(\beta)$ such that $\gamma< 1-2\beta$ and $N\leq O(n^{\gamma})$.
Consider $X_N$ as before. Then for $p=c/(n-2k+1)$, $c,k\in\bbN$, $$\bbP(X_N>0)\rightarrow0$$ as $n\rightarrow\infty$. 
\end{lem} 
\begin{proof}
By  Markov's inequality, $$\bbP(X_N>0)\leq \bbE(X_N),$$  and so it will be enough to show that $\bbE(X_N)\rightarrow 0$ as $n\rightarrow\infty$.
Let $L_{ij}$ be the number of shortcuts between two disjoint blocked clusters of sizes $i$ and $j$, with $i,j< O(n^{\beta})$ (see Proposition \ref{prop:maxsizeJv}). Note that $L_{ij}$ follows a Binomial distribution with parameters $(ij,p)$. Hence,
\begin{align}
 \bbP(L_{ij}>0)=&1-(1-p)^{ij},\text{ and} \label{pij}\\
 \bbP(L_{ij}\geq2)=&1-(1-p)^{ij}-ijp(1-p)^{ij-1} \label{2pij}.
 \end{align}

Using  Taylor expansion for $e^x$ and $\ln(1-x)$, $|x|<1$, we get that for $a\geq0$
 \begin{equation}\label{pij2}
 (1-p)^{ij-a}=e^{(ij-a)\ln(1-p)}=1-(ij-a)p+O\Big(\Big(\frac{ij-a}{n-2k+1}\Big)^2\Big).
 \end{equation}
Replacing (\ref{pij2}) in (\ref{pij}) and (\ref{2pij}), we obtain
 \begin{align}
 \bbP(L_{ij}>0)=& ijp+O\Big(\Big(\frac{ij}{n-2k+1}\Big)^2\Big)\leq ijp+ O(n^{2(2\beta-1)}), \text{ and}\label{pij3}\\
  \bbP(L_{ij}\geq2)=& O\Big(\Big(\frac{ij}{n-2k+1}\Big)^2\Big)\leq O(n^{2(2\beta-1)}).\label{2ormore}
 \end{align} 
Now let $\mathcal{S}_m$ be the set of all subsets of $m$ disjoint blocked clusters, $2\leq m\leq N$, connected by shortcuts and ordered up to rotation and orientation of the cycle. Let $\mathcal{S}=\cup_{m=2}^N \mathcal{S}_m$. For each $S_m\in \mathcal{S}_m$, define $A_{S_m}$ to be the event that a cycle occurs between the $m$ disjoint blocked clusters of $S_m$. As the expectation is linear, 
\begin{eqnarray}\label{esperanza}
\bbE(X_N)=\sum_{S\in\mathcal{S}}\bbE(1_{A_S})&=&
\sum_{m=2}^N\sum_{S_m\in\mathcal{S}_m}\bbP(A_{S_m}).
\end{eqnarray}

Observe that for $m=2$, $A_{S_m}$ means that two disjoint blocked clusters  are connected through two or more shortcuts, since the vertices inside the disjoint blocked clusters are connected through local edges in $\mathcal{G}(n,k,p)$. Otherwise, for $m\geq 3$, $A_{S_m}$ means that there is at least $m$ independent shortcuts, each joint two disjoint blocked clusters. Therefore, by (\ref{2ormore})
\begin{equation}\label{A2}
\bbP(A_{S_2})\leq O(n^{2(2\beta-1)}),
\end{equation}
while for $m\geq 3$ and assuming that $S_m$ is formed by disjoint blocked clusters of sizes $i_1,i_2,\dots,i_m$, we have by (\ref{pij3})  
\begin{align}\label{Am}
\bbP(A_{S_m})\leq[i_1i_2p+ O(n^{2m(2\beta-1)})]\dots[i_mi_1p+ O(n^{2m(2\beta-1)})]
 \leq& O(n^{m(2\beta-1)})).
\end{align} 
In addition we need to find $|\mathcal{S}_m|$. The number of ordered sets of size $m$ is $\binom{N}{m}m!$, which over-counts each $S_m\in\mathcal{S}_m$, $m\geq3$, by $2m$ times (once for each starting position on the cycle $(\times m)$, and once for each direction of the cycle $(\times2)$), while for $m=2$ over-counts only by 2 times (just for each direction of the cycle). Thus by (\ref{esperanza}), (\ref{A2}) and (\ref{Am})
\begin{eqnarray*}
\bbE(X_N)&\leq&\binom{N}{2}O(n^{2(2\beta-1)})+\sum_{m=3}^N \binom{N}{m}\frac{m!}{2m}O(n^{m(2\beta-1)})\\
&\leq& \sum_{m=2}^N \Big[N O(n^{(2\beta-1)}))\Big]^m.
\end{eqnarray*}
Since $N\leq O(n^{\gamma})$, $\gamma<1-2\beta$ and $0<\beta<1/2$, then $NO(n^{2\beta-1})<1$ and goes to zero as $n\rightarrow\infty$. Hence, by the geometric series
$$
\bbE(X_N)<\frac{[NO(n^{2\beta-1})]^2}{1-NO(n^{2\beta-1})}\rightarrow_{n\rightarrow\infty}0.
$$
\end{proof}

\begin{lem}\label{necesaria}
Let $N\leq O(n^{\gamma})$ and let $u,v\in\RR(\tau_n)$, such that, $\{v\in \mathcal{N}_s(u)\}$ and for some $t\leq {t_N}$ $\{u\in \SS(t)\}\cap \{U_t=u\}$. Then  either $\{v\in\II(t)\}$ or $\{v\in \RR(t)\cup\SS(t)\text{ and } pred(v)=u\}$ $w.h.p$.
 \end{lem}  
In words, what Lemma \ref{necesaria} tell us is that if at some time $t\leq t_N$, a spreader $u$ is chosen to inform,  then $w.h.p$ no neighbour $v$ connected to $u$ through shortcuts has already been informed by other  spreader. 

\begin{proof}
Let $A$ be the event that the $N$ disjoint $T\ell c$'s, $\mathfrak{C}_{v_{I_h}}^T(t_N)$, $h=0,\dots,N-1$, are connected by at most $N-1$ shortcut, each joint two distinct $T\ell c$'s, and $E$ the event  $\{\mathfrak{C}_{v_{I_h}}^T(t_N)\subseteq \mathfrak{B}_{v_{I_h}}, h=0,\dots,N-1\}$.

By Lemma \ref{lem:local-blocked-cluster} we know that $\CC_v(\tau_n) \subseteq B_v$ $w.h.p$, where $B_v$ is the blocked cluster of $v$ ($|B_v|< O(n^{\beta})$, $\beta<1/2$ $w.h.p$ by Proposition \ref{prop:maxsizeJv}). Then, it implies that for $h\geq0$,
$
\mathfrak{C}_{v_{I_h}}^T(t_N)\subseteq \mathfrak{B}_{v_{I_h}}
$ $w.h.p$. Thus, we have that $E$ holds $w.h.p$.

Now note that if $A$ holds, then up to time $t_N$, we see a tree of disjoint $T\ell c$'s. Therefore, if  $u,v\in\RR(\tau_n)$ such that, $\{u\in \SS(t)\}\cap \{U_t=u\}$ and $\{v\in \mathcal{N}_s(u)\}$, $t\leq t_N$, then up to time $t_N$, $v$ may only be informed by $u$. Thus, $A$ implies that
\begin{equation}\label{Aimplies}
\{v\in\II(t)\}\cup \{v\in \RR(t)\cup\SS(t)\text{ and } pred(v)=u\}.
\end{equation}

Moreover, since $E$ holds $w.h.p$, then conditioning on this we get
\begin{align}
\bbP(A)=& \bbP(A\mid E)+o(1)\nonumber\\
\geq& \bbP(X_N=0)+o(1),\label{P(A)}
\end{align}
where $X_N$ is the number of cycles between the disjoint blocked clusters $\{\mathfrak{B}_{v_{I_h}}\}_{h=0}^{N-1}$.

Thus, by (\ref{Aimplies}) and (\ref{P(A)}) we have
\begin{eqnarray*}
\bbP\Big(\{v\in\II(t)\} \cup \{v\in \RR(t)\cup\SS(t)\text{ and } pred(v)=u\}\Big)&\geq &\bbP(X_N= 0)+o(1).
\end{eqnarray*}
Finally, by Lemma \ref{existenceTv0} up to time $t_N$,
$\bbP(X_N= 0)\rightarrow 1$ as $n\rightarrow\infty$. Hence,  for any $t\leq t_N$
\begin{eqnarray*}
\bbP\Big(\{v\in\II(t)\} \cup \{v\in \RR(t)\cup\SS(t)\text{ and } pred(v)=u\}\Big)\rightarrow 1,
\end{eqnarray*}
as $n\rightarrow\infty$.
\end{proof}
\subsubsection{A super critical branching process}\label{supercritical}
We start this section by exploring, up to time $t_N$, the set $\kappa_{v_0}(\tau_n)$. To pursue this aim we consider the following random times together with the next lemma. For any $v\in [n]$ let
\begin{equation}\label{eq:random_times}
t_v^s:=\inf\{t\geq 0: v\in \mathcal{S}(t)\};\,\,\,
t_{v}':=\inf\{t> t_v^s:U_t =v\};\,\,\,
t_{v}'':=\inf\{t> t_{v}':U_t =v\},
\end{equation}  
and we say that these values are $\infty$, if the respective infimum is not attained. In words, $t_v^s, t_{v}'$ and $t_{v}''$ are the times at which the vertex $v$ becomes a spreader, attempts to transmit the information for the first time, and attempts to transmit the information for the second time, respectively. Observe that $t_{v_0}^s=0$ and $t_{v_0}'$, $t_{v_0}''$ are finite. On the other hand, if $v\in \RR(\tau_n)$ then $t_v^s$ and $t_{v}'$ are finite by definition and construction, but this is not necessarily true for $t_{v}''$ (for example, we could have $U_v^{1}\in  \RR(t_{v}')\cup \SS(t_{v}')$). 

\begin{lem}\label{timetwo}
Let $t'_{v_{I_h}}$ and $t''_{v_{I_h}}$ be the first and the second time, respectively,  at which $v_{I_h}$ attempts to transmit the information, and defined by  \eqref{eq:random_times}. Then, $t''_{v_{I_0}}<\infty$, while for $h\geq1$, and as long as $t'_{v_{I_h}}\leq t_N$, $t''_{v_{I_h}}<\infty$ $w.h.p$. 
\end{lem}
\begin{proof}
Since at time $t'_{v_{I_0}}\equiv1$,  all the vertices except $v_{I_0}$ are ignorants, then $t''_{v_{I_0}}<\infty$.  Moreover, for $v\in\mathcal{N}_s(v_{I_h})$, $h\geq 1$, and such that $t'_{v_{I_h}}\leq t_N$,
by Lemma \ref{necesaria} we have that either $\{v\in\II(t'_{v_{I_h}})\}$ or $\{v\in \RR(t'_{v_{I_h}})\cup\SS(t'_{v_{I_h}})\text{ and } pred(v)=v_{I_h}\}$ $w.h.p$. Observe that this is the same to say $\{v\in\II(t'_{v_{I_h}})\}$ $w.h.p$, since $t'_{v_{I_h}}$ is the first time when  $v_{I_h}$ is selected to inform. Therefore, what we have is  that $t''_{v_{I_h}}<\infty$ $w.h.p$, as long as $t'_{v_{I_h}}\leq t_N$.  
\end{proof}

Now we are ready to explore, up to time $t_N$, $\kappa_{v_0}(\tau_n)$:

\begin{enumerate}
\item Begin with $v_0$ and find all the vertices connected to $v_0$ by shortcuts, i.e, $\mathcal{N}_s(v_0)$.  These vertices coincide with the set $\mathcal{E}_1$ defined in Section \ref{sec:construction}. 

Consider $t_{v_0}'$ and $t_{v_0}''$, they defined by \eqref{eq:random_times}. At these times, either one or two of all possible vertices incident to $v_0$ are informed according to the Uniform random variables $U_{v_0}^1$ and $U_{v_0}^2$, respectively. Take account only the vertices informed at times $t_{v_0}'$ and $t_{v_0}''$ through shortcuts, and let $Z_{I_0}$ be the random variable corresponding to this number of vertices.
\item Given that $v_0$ informs a vertex  in $\mathcal{N}_s(v_0)$ at time   $t_{v_0}'$ or $t_{v_0}''$, e.g. $u^1$ and $u^2$, let $t_{u^j}'$ and $t_{u^j}''$ (assume these times are finite), $j=1,2$, be the first and the second time  when $u^j$ attempts to transmit the information, respectively. If $t_{u^j}'\leq t_N$, then  find all the vertices connected to $u^j$ by shortcuts, but different of $v_0$, i.e., $\mathcal{N}_s(u^j)\backslash \{v_0\}$. 
Take account only the vertices informed at times $t_{u^j}'$ and $t_{u^j}''$ through shortcuts.
\item Repeat step (2) for $u^j$, $j=1,2$, and so on, provided that the first time at which a vertex attempts  to transmit the information, is less than $t_N$. 
\end{enumerate}
Denote this sub-cluster by $\kappa_{v_0}(t_N)$ and note that $|\kappa_{v_0}(t_N)|\leq N$.   

\subsubsection{Proof of Theorem \ref{teor:phaset} (ii)}
We start by observing that the number of vertices at the end of each step in the previous exploration process of  $\kappa_{v_0}(t_N)$, corresponds to the  number of vertices informed by some $v_{I_h}$ through shortcuts, at the first and the second time at which $v_{I_h}$ attempts to transmit the information.
For each $v_{I_h}$, $h\geq0$ and as long as $t_{v_{I_h}}'\leq t_N$, let $Z_{I_h}$ denotes this  number of vertices.

We define now a random variable  $X_{I_h}$ such that it is stochastically dominated by  $Z_{I_h}$, for each $h\geq0$ and as long as $t'_{v_{I_h}}\leq t_N$. That is, $\bbP(Z_{I_h}\geq z)\geq \bbP(X_{I_h}\geq z)$ for all $z$, with strict inequality at some $z$. To this aim we begin by observing that  the event $\{Z_{I_h}=1\}$ holds, if and only if, any of the two following possibilities hold.

\begin{enumerate}
\item[$P_1^h$:] At times  $t'_{v_{I_h}}$ and $t''_{v_{I_h}}$, $v_{I_h}$ chooses the same neighbor and it belongs to  $\mathcal{N}_s(v_{I_h})$, or
\item[$P_2^h$:] at time  $t'_{v_{I_h}}$, $v_{I_h}$ chooses a neighbor from $\mathcal{N}_l(v_{I_h})$ and at time $t''_{v_{I_h}}$, $v_{I_h}$ chooses a neighbor from $\mathcal{N}_s(v_{I_h})$.
\end{enumerate}     
Consider only the first posibility. Since  at time $t'_{v_{I_0}}$  all the vertices except $v_{I_0}$ are ignorants, then for $h=0$,  
\begin{eqnarray}\label{unocero}
\bbP(Z_{I_0}=1)&>&\sum_{\ell=1}^{n-2k-1}\bbP\Big(\{U_{v_{I_0}}^1=U_{v_{I_0}}^2=v \}\cap\{v\in \mathcal{N}_s(v_{I_0})\}\cap\{v\in\II(t'_{v_{I_0}})\}\cap\{|\mathcal{N}_s(v_{I_0})|=\ell\}\Big)\nonumber\\
&=&\sum_{\ell=1}^{n-2k-1}\Big(\frac{\ell}{(\ell+2k)^2}\Big)\bbP(|\mathcal{N}_s(v_{I_0})|=\ell)\nonumber\\
&>&\sum_{\ell=1}^{n-2k-2}\Big(\frac{\ell}{(\ell+2k+1)^2}\Big)\bbP(|\mathcal{N}_s(v_{I_0})|=\ell),
\end{eqnarray}
where $|\mathcal{N}_s(v_{I_0})|$ follows a Binomial distribution with parameters $(n-2k-1,p)$.

On the other hand, observe that the event $\{Z_{I_0}=2\}$ holds, if and only if,  at time $t'_{v_{I_0}}$, $v_{I_0}$ chooses a neighbor from $\mathcal{N}_s(v_{I_0})$ and at time $t''_{v_{I_0}}$, $v_{I_0}$ chooses another neighbor from $\mathcal{N}_s(v_{I_0})$ but different of the first one. Hence,
\begin{eqnarray}\label{uno}
\bbP(Z_{I_0}=2)&=&\sum_{\ell=2}^{n-2k-1}\bbP\Big(\{U_{v_{I_0}}^1=v,U_{v_{I_0}}^2=w,
v\neq w\}\cap\{v,w\in \mathcal{N}_s(v_{I_0})\}\cap\{v\in\II(t'_{v_{I_0}})\}\cap\{|\mathcal{N}_s(v_{I_0})|=\ell\}\Big)\nonumber\\
&=&\sum_{\ell=1}^{n-2k-1}\Big(\frac{\ell(\ell-1)}{(\ell+2k)^2}\Big)\bbP(|\mathcal{N}_s(v_{I_0})|=\ell)\nonumber\\
&>&\sum_{\ell=1}^{n-2k-2}\Big(\frac{\ell(\ell-1)}{(\ell+2k+1)^2}\Big)\bbP(|\mathcal{N}_s(v_{I_0})|=\ell).
\end{eqnarray}
For $h>0$ and as long as $t'_{v_{I_h}}\leq t_N$, by Lemma \ref{timetwo} we know that  $t''_{v_{I_h}}<\infty$ $w.h.p$. Conditioning on this event, and following a similar reasoning to get (\ref{unocero}) and (\ref{uno}), we obtain 

\begin{eqnarray}\label{dos}
\bbP(Z_{I_h}=1)&>&
\sum_{\ell=1}^{n-2k-2}\Big(\frac{\ell}{(\ell+2k+1)^2}\Big)\bbP(|\mathcal{N}_s(v_{I_h})|=l)+o(1),
\end{eqnarray}
and
\begin{eqnarray}\label{dosdos}
\bbP(Z_{I_h}=2)&>&
\sum_{\ell=2}^{n-2k-2}\Big(\frac{\ell(\ell-1)}{(\ell+2k+1)^2}\Big)\bbP(|\mathcal{N}_s(v_{I_h})|=l)+o(1),
\end{eqnarray}
where  $|\mathcal{N}_s(v_{I_h}|$ follows a Binomial distribution with parameters $(n-2k-2,p)$. 

Using (\ref{unocero}) and  (\ref{uno}), let $X_{I_0}$ be a random variable with probability density given by
\begin{eqnarray}
\bbP(X_{I_0}=1)=\sum_{\ell=1}^{n-2k-2}\Big(\frac{\ell}{(\ell+2k+1)^2}\Big)\bbP(|\mathcal{N}_s(v_{I_0})|=l),\nonumber\\
\bbP(X_{I_0}=2)=\sum_{\ell=2}^{n-2k-2}\Big(\frac{\ell(\ell-1)}{(\ell+2k+1)^2}\Big)\bbP(|\mathcal{N}_s(v_{I_0})|=l)\text{ and}\nonumber\\
\bbP(X_{I_0}=0)=1-\bbP(X_{I_0}=1)-\bbP(X_{I_0}=2).\nonumber
\end{eqnarray}

In analogous way but using  (\ref{dos}) and (\ref{dosdos}), define for each $h>0$ and as long as $t'_{v_{I_h}}\leq t_N$,  the random variables $X_{I_h}$'s. By definition we observe that each random variable  $X_{I_h}$ is stochastically dominated by  $Z_{I_h}$.

Observe that all the random variables $X_{I_h}$ and $Z_{I_h}$ depend on $n$. Furthermore, for each $h\geq 0$ and as long as $t'_{v_{I_h}}\leq t_N$,
\begin{eqnarray}\label{limitXIh}
\lim_{n\rightarrow\infty}\bbP(X_{I_h}=1)&=&\bbE\left(\frac{X}{(X+2k+1)^2}\right),\nonumber\\
\lim_{n\rightarrow\infty}\bbP(X_{I_h}=2)&=&\bbE\left(\frac{X(X-1)}{(X+2k+1)^2}\right)\text{ and}\nonumber\\
\lim_{n\rightarrow\infty}\bbP(X_{I_h}=0)&=&1-\bbE\left(\frac{X^2}{(X+2k+1)^2}\right),\nonumber\\
\end{eqnarray}
where $X$ follows a Poisson distribution with mean $c$, and where we have used the Poisson aproximation to the Binomial distribution (see \cite{durrettbook} Theorem 6.1). 

Now, we construct a growth process in a similar way as we explored $\kappa_{v_0}(t_N)$ in subsection \ref{supercritical}, but using the random variables $X_{I_h}$ instead of $Z_{I_h}$. We denote this process by $\mathcal{Y}(t_N)$ and let $| \mathcal{Y}(t_N)|$ be its size.  

Since for each $h\geq 0$ and as long as $t'_{v_{I_h}}\leq t_N$, $Z_{I_h}$ stochastically dominates $X_{I_h}$, then we  have that the random variable $| \mathcal{Y}(t_N)|$ is stocastically dominated by $|\kappa_{v_{I_0}}(t_N)|$, that is
\begin{equation}\label{YK}
\bbP\left( |\kappa_{v_0}(t_N)|\geq L\right)\geq\bbP\left(|\mathcal{Y}(t_N)|\geq L\right),
\end{equation}
for all $L>0$, with stric inequality at some $L$.
Furthermore, by construction we have that
\begin{equation}\label{KR}
\RR(\tau_n)\supseteq \RR(t_N)\supseteq \kappa_{v_0}(t_N).
\end{equation} 

Take $L\leq N$ but of the same order. That is, $L:=L(N)=O(n^{\gamma})$. Then, by (\ref{YK}) and (\ref{KR}) we have that
\begin{eqnarray}\label{430}
\bbP(|\RR(\tau_n)|\geq L)&\geq&\bbP(|\RR(t_N)|\geq L)\nonumber\\
&\geq&\bbP(|\kappa_{v_0}(t_N)|\geq L)\nonumber\\
&\geq&\bbP(|\mathcal{Y}(t_N)|\geq L).
\end{eqnarray}

Finally, let $(Y_s)_{s\geq0}$ be a branching process with offspring distribution given by the limit  distribution of $X_{I_0}$. Denote by $|Y|$ the size of the total progeny of $(Y_s)_{s\geq0}$.  Using (\ref{limitXIh}), this branching process has mean offspring distribution 
\begin{equation}\label{mYs}
\tilde m(k,c):=\bbE\left(\frac{X}{(X+2k+1)^2}\right)+2\bbE\left(\frac{X(X-1)}{(X+2k+1)^2}\right).
\end{equation}
Therefore, by (\ref{430})
\begin{eqnarray*}
\liminf_{n\rightarrow\infty}\bbP(|\RR(\tau_n)|\geq L)\geq \liminf_{n\rightarrow\infty}\bbP(|\mathcal{Y}(t_N)|\geq L)&=&\lim_{n\rightarrow\infty}\bbP(|\mathcal{Y}(t_N)|\geq L)\\
&\geq& \bbP((Y_s)_{s\geq0} \text{ survives}).\\
\end{eqnarray*}

Using the theory of branching process, $(Y_s)_{s\geq0}$ survives with positive probability if $\tilde m(k,c)>1$, see \cite{BranchingProcesses}.  
Finally, we can choose $c$ sufficiently large in such a way that $\tilde m(k,c)>1$. This can be justified by noting, after some calculations, that as a function of $c$, $\tilde m(k,c)$ is right-continuous at $[0,\infty)$, non-decreasing, and $\tilde m(k,0)=0$. 


\section{Appendix B - Technical details on numerical simulations}
In this appendix we deepen and compare the two sources of stochasticity present in the process under consideration.
In fact, on the one hand, the graph $\mathcal{G}(n,k,p)$ underlying the diffusion process is stochastic, since the $nc$ new links to be inserted are drawn randomly. On the other hand, the dynamic process is intrinsically stochastic, since at each step the spreader and its neighbour are extracted randomly. Thus, one can in principle proceed in one of the following ways:\\
Dynamical noise:
\begin{itemize}
	\item Fix $k$ and take a $k$-ring of size $n$;
	\item Fix $c$ and insert new edges with probability $p=c/(n-2k+1)$, hence getting the realization $G_n$ for the graph $\mathcal{G}(n,k,p)$;
	\item Perform $M$ cycles of the Maki-Thomson dynamic, and for each of them, say the one labelled as $i$ ($i=1,...,M$), calculate the final value of removed individuals $R^{(i)}(\tau_n,G_n)$;
	\item Compute the mean value over the $M$ realizations of the dynamics as
\be	
\langle R(c, k,  G_n )\rangle_{dyn} =\sum_{i=1}^{M}\frac{R^{(i)}(\tau_n,G_n)}{M}.
\ee
	\item Repeat the procedure for different values of $c$ and $k$.
\end{itemize}

Notice that this procedure does not take into account the noise due to the stochastic realization of the graph. \\
Topological noise:
\begin{itemize}
	\item Fix $k$ and take a $k$-ring of size $n$;
	\item Fix $c$ and insert new edges with probability $p=c/(n-2k+1)$, hence getting the realization $G_n^{(i)}$ for the graph $\mathcal{G}(n,k,p)$;
	\item Perform the Maki-Thomson dynamic and calculate the final value of removed individuals $R(\tau_n,G_n^{(i)})$
	\item Build $L$ independent realizations of the graph and for each of them, say the one labelled as $i$ ($i=1,...,L$) and referred to as $G^{(i)}_n$, perform a cycle of the dynamic, and calculate the final value of removed individuals $R(\tau_n, G^{(i)}_n )$. 
	\item Compute the mean value over the $L$ realizations of the graph as
	\be
	 \langle R(c, k) \rangle_G=\sum_{i=1}^{L}\frac{R(\tau_n, G^{(i)}_n)}{L}.
	\ee
	\item Repeat the procedure for different values of $c$ and $k$.
\end{itemize}

Notice that in this procedure each realization of the dynamics corresponds to a different realization of $\mathcal{G}(n,k,p)$.

We compared the outcomes from these experiments when the same number of items is averaged (i.e., $M=L$). First, we notice that, the average values stemming from both routes are comparable within the related errors.
Moreover, the second route turns out to be more noisy; this result perfectly matches with the fact that, in a Monte Carlo-like simulation, means computed on different graph are less accurate than the ones related to a single realization of network (see e.g., \cite{ABC-JSP}). 
\newline
Anyhow, the numerical path followed in Sec.~\ref{sec:numerics} merges these two routes by first applying the former route and then repeating the procedure over different realizations of the underlying graph and finally averaging.

\section*{Acknowledgments}
\noindent
PMR thanks FAPESP (Grants 2015/03868-7 and 2016/11648-0) for financial support. Part of this work was carried out during a stay of PMR at Laboratoire de Probabilit\'es et Mod\`eles Al\'eatoires, Universit\'e Paris-Diderot, and a visit at Universit\`a di Torino. He is grateful for their hospitality and support.\\
AP thanks Universit\`a di Torino (XVIII tornata Programma di ricerca: ``Problemi attuali della matematica 3") for financial support.\\
EA and FT thank INdAM-GNMF (Progetto Giovani 2016) and Sapienza Universit\`a di Roma (Progetto Avvio alla Ricerca 2015) for financial support.\\

\end{document}